\documentclass{amsart}

\usepackage{amsthm,amsmath,amssymb,amsxtra,amscd,enumerate}
\usepackage[all]{xy}
\usepackage{hyperref}
\numberwithin{equation}{section}

\theoremstyle{plain}
\newtheorem{thm}{Theorem}[section]
\newtheorem{prop}[thm]{Proposition}
\newtheorem{lem}[thm]{Lemma}

\newtheorem{fact}[thm]{Fact}
\theoremstyle{definition}
\newtheorem{de}[thm]{Definition}
\newtheorem{rem}[thm]{Remark}

\newcommand{\bb}{\mathbb}
\newcommand{\cal}{\mathcal}
\newcommand{\ovl}{\overline}
\newcommand{\wtl}{\widetilde}

\DeclareMathOperator*{\Hom}{Hom}
\DeclareMathOperator*{\pr}{pr}
\DeclareMathOperator*{\Ass}{Ass}
\DeclareMathOperator{\Ad}{Ad}
\DeclareMathOperator{\ad}{ad}
\DeclareMathOperator{\Int}{Int}

\title[Restriction of Zuckerman's derived functor modules]
 {On the restriction of Zuckerman's derived functor modules
 $A_{\mathfrak{q}}(\lambda)$ to reductive subgroups}

\thanks{%
}

\author{Yoshiki Oshima}

\address[Yoshiki Oshima]{Graduate School of Mathematical Sciences, 
        The University of Tokyo, 3-8-1 Komaba, Meguro, 153-8914 Tokyo, Japan}

\email{yoshiki@ms.u-tokyo.ac.jp}


\begin{document}

\maketitle

\begin{abstract}
In this article, 
 we study the restriction of Zuckerman's derived functor
 $(\mathfrak{g},K)$-modules $A_{\mathfrak{q}}(\lambda)$
 to $\frak{g}'$
 for symmetric pairs of reductive Lie algebras $(\mathfrak{g}, \mathfrak{g}')$.
When the restriction decomposes into irreducible
 $(\mathfrak{g}',K')$-modules, we give an upper bound for the branching law.
In particular, we prove that each $(\frak{g}',K')$-module occurring in the
 restriction is isomorphic to a submodule of $A_{\mathfrak{q}'}(\lambda')$
 for a parabolic subalgebra
 $\mathfrak{q}'$ of $\mathfrak{g}'$, and
 determine their associated varieties.
For the proof, we construct $A_{\mathfrak{q}}(\lambda)$
 on complex partial flag varieties by using ${\cal D}$-modules.
\end{abstract}

\section{Introduction}
Our object of study is branching laws of
 Zuckerman's derived functor modules $A_\frak{q}(\lambda)$
 with respect to symmetric pairs of real reductive Lie groups.

Let $G_0$ be a real reductive Lie group with Lie algebra $\frak{g}_0$.
Fix a Cartan involution $\theta$ of $G_0$ so that
 the fixed set $K_0:=(G_0)^\theta$ is a maximal compact subgroup
 of $G_0$.  
Write $K$ for the complexification of $K_0$, 
 $\frak{g}_0=\frak{k}_0\oplus\frak{p}_0$ for
 the Cartan decomposition with respect to $\theta$
 and $\frak{g}:=\frak{g}_0\otimes_{\bb{R}} \bb{C}$
 for the complexification.
Similar notation will be used for other Lie algebras.
The cohomologically induced module $A_\frak{q}(\lambda)$ is
 a $(\frak{g},K)$-module
 defined for a $\theta$-stable parabolic subalgebra $\frak{q}$
 of $\frak{g}$ and a character $\lambda$.
The $(\frak{g},K)$-module $A_\frak{q}(\lambda)$ is unitarizable
 under a certain condition on the parameter $\lambda$
 and therefore plays a large part in the study
 of the unitary dual of real reductive Lie groups.

One of the fundamental problems in the representation theory
 is to decompose a given representation into irreducible constituents.
To begin with, we consider the restriction of $(\frak{g},K)$-modules
 to $K$, or equivalently, to the compact group $K_0$.
In this case, any irreducible $(\frak{g},K)$-module
 decomposes as the direct sum of irreducible representations of $K$
 and each $K$-type occurs with finite multiplicity.
For $A_\frak{q}(\lambda)$, 
 the following formula gives an upper bound for the multiplicities.

\begin{fact}[{\cite[\S V.4]{KnVo}}]
\label{upblattner}
Let $\frak{u}$ be the nilradical of $\frak{q}$.
Take a Cartan subalgebra $\frak{t}_0$ of $\frak{k}_0$
 such that $\frak{t}\subset \frak{q}\cap \frak{k}$
 and choose a positive system $\Delta^+(\frak{k},\frak{t})$
 contained in $\Delta(\frak{q}\cap \frak{k},\frak{t})$.
For a dominant integral weight $\mu\in\frak{t}^*$
 write $F(\mu)$ for the irreducible finite-dimensional representation
 of $K$ with highest weight $\mu$.
Then
\begin{align}
\label{introinjk}
A_\frak{q}(\lambda)|_{K}
\leq 
\bigoplus_{p=0}^{\infty}
\bigoplus_{\mu}
F(\mu)^{\oplus m(\mu,\,p)},
\end{align}
where $m(\mu,p)$ is the multiplicity
 of weight $\mu$ in
 $\bb{C}_{\lambda+2\rho(\frak{u}\cap\frak{p})}\otimes
 S^p(\frak{u}\cap \frak{p})$.
\end{fact}
There is also an explicit branching formula
 of $A_{\frak{q}}(\lambda)|_{K}$ for weakly fair $\lambda$,
 known
 as the generalized Blattner formula
 (see \cite[\S II.7]{Bien}, \cite[\S V.5]{KnVo}).

On the other hand, the restriction to a non-compact subgroup
 is more complicated.
Let $\sigma$ be an involution of $G_0$ that commutes with
 $\theta$ and let
 $G'_0$ be the identity component of $(G_0)^\sigma$.
The pair $(G_0,G'_0)$ is called a symmetric pair.
Write $\frak{g}'$ for the complexified Lie algebra
 of $G'_0$ and write $K'$ for the complexification of
 the maximal compact group $K'_0:=(G'_0)^\theta$ of $G'_0$.
If $G'_0$ is non-compact,
 the restriction $A_\frak{q}(\lambda)|_{(\frak{g}',K')}$
 does not decompose into irreducible
 $(\frak{g}',K')$-modules in general.
Indeed, $A_\frak{q}(\lambda)|_{(\frak{g}',K')}$
 does not have any irreducible submodule in many cases.

Nevertheless, there are classes of
 $(\frak{g},K)$-modules which decompose
 into irreducible $(\frak{g}',K')$-modules and
 explicit branching formulas were obtained
 for some particular representations
 \cite{DuVa}, \cite{GW}, \cite{kob93}, \cite{kob94},
 \cite{KoOr}, \cite{Lo}, \cite{OrSp}, \cite{Se}.
In his series of papers \cite{kob93}, 
 \cite{kob94}, \cite{kob98i}, \cite{kob98ii},
 Kobayashi introduced the notion of
 discretely decomposable $(\frak{g}',K')$-modules
 and gave criteria for the discretely decomposable restrictions
 (see Fact~\ref{discdecomp}).
By virtue of this result,
 we can single out $A_\frak{q}(\lambda)$ 
 that decompose into
 irreducible $(\frak{g}',K')$-modules. 
See \cite{KoOs} for a classification
 of the discretely decomposable
 restrictions $A_\frak{q}(\lambda)|_{(\frak{g}',K')}$.
Recent developments on these subjects
 are discussed in \cite{kob11}.

Our aim is to find a branching law of
 $A_\frak{q}(\lambda)|_{(\frak{g}',K')}$
 when it is discretely decomposable.
The main result of this article is Theorem~\ref{branchup2},
 where we construct an injective $(\frak{g}',K')$-homomorphism:
\begin{align}
\label{introinj}
A_\frak{q}(\lambda)\to
\bigoplus_{p=0}^{\infty}
\bigoplus_{\lambda'}
A_{\frak{q}''}(\lambda')^{\oplus m(\lambda'\!,\,p)}.
\end{align}
The parabolic subalgebra $\frak{q}''$ of $\frak{g}'$ and 
 the multiplicity function $m(\lambda', p)$ are given
 in \eqref{defq''} and \eqref{defmult}, respectively.
Theorem~\ref{branchup2} is a generalization of 
 Fact~\ref{upblattner} because
 if $\theta=\sigma$, then $G'_0=K_0$ and
 it turns out that
 the right side of \eqref{introinj} is isomorphic to
 the right side of \eqref{introinjk} as a $K$-module.

For the proof of the theorem, we realize $A_\frak{q}(\lambda)$
 as the global sections of sheaves on complex partial flag varieties
 in Theorem~\ref{locZuc},
 using ${\cal D}$-modules.
A relation
 between cohomologically induced modules
 and twisted ${\cal D}$-modules on the complete flag variety
 was constructed by
 Hecht--Mili$\rm{\check{c}}$i$\rm{\acute{c}}$--Schmid--Wolf~\cite{dual}.
See \cite{Bien},
\cite{Kit}, \cite{MiPa} for further developments of this result. 
Our proof of Theorem~\ref{locZuc} is based on \cite{dual}. 

As a corollary to Theorem~\ref{branchup2}, 
 we determine the associated varieties of
 the irreducible constituents of
 $A_\frak{q}(\lambda)|_{\frak{g}'}$
 in Theorem~\ref{branchass}.
Let $W$ be an irreducible $(\frak{g},K)$-module and
 $V$ an irreducible $(\frak{g}',K')$-module such that
 $\Hom_{\frak{g}'}(V,W)\neq 0$.
Write ${\pr}_{\frak{g}\to\frak{g}'}:\frak{g}^*\to(\frak{g}')^*$
 for the restriction map 
 and write ${\Ass}_{\frak{g}'}(V), {\Ass}_{\frak{g}}(W)$
 for the associated varieties of $V, W$, respectively.
Then the inclusion
 ${\pr}_{\frak{g}\to\frak{g}'}({\Ass}_{\frak{g}}(W))\subset
 {\Ass}_{\frak{g}'}(V)$ was proved in \cite{kob98ii}
 and the equality 
\begin{align}
\label{asseq}
{\pr}_{\frak{g}\to\frak{g}'}({\Ass}_{\frak{g}}(W))=
{\Ass}_{\frak{g}'}(V)
\end{align}
was conjectured in \cite{kob11}.
Using Theorem~\ref{branchup2}, we show that the equality \eqref{asseq} holds
 for $W=A_\frak{q}(\lambda)$.

This article is organized as follows.
In Section \ref{sec:coh}, we recall the definitions of cohomological induction
 and $A_\frak{q}(\lambda)$, following the book by Knapp--Vogan~\cite{KnVo}.
In this article, we extend actions of a compact group $K_0$
 to actions of its complexification $K$, and
 view $(\frak{g},K_0)$-modules as $(\frak{g},K)$-modules.
In Section~\ref{sec:homog}, we fix notation 
 and prove lemmas concerning
 homogeneous spaces and differential operators.
Lemma~\ref{affloc} is used in the proof of Theorem~\ref{locZuc}.
Section~\ref{sec:loc} is devoted to the proof of
 Theorem~\ref{locZuc}.
In Section~\ref{sec:constparab}, we construct
 $\theta$-stable parabolic subalgebras of $\frak{g}'$
 that will appear in the branching laws, using a
 criterion for the discrete decomposability given in \cite{kob98ii}.
The parabolic subalgebra $\frak{q}''$ is defined in \eqref{defq''}. 
We prove Theorem~\ref{branchup2} in
 Section~\ref{sec:branch}. 
We study the associated varieties in Section~\ref{sec:ass}.

{\bf Acknowledgements.}
I am deeply grateful to my advisor
 Professor Toshiyuki Kobayashi
 for his helpful comments and warm encouragement.
I am supported by the Research Fellowship of the
 Japan Society for the Promotion of Science for Young Scientists.


\section{Cohomological Induction}\label{sec:coh}
In this section, we fix notation concerning cohomological induction
 and $A_\frak{q}(\lambda)$, following \cite{KnVo}.

Let $K_0$ be a compact Lie group.
The complexification $K$ of $K_0$ has the structure of
 reductive linear algebraic group.
Since any locally finite action of $K_0$
 is uniquely extended to an algebraic
 action of $K$, the locally finite $K_0$-modules are identified with
 the algebraic $K$-modules.

Define the Hecke algebra $R(K_0)$ as the space of
 $K_0$-finite distributions on $K_0$. 
For $S\in R(K_0)$, the pairing with a smooth function $f\in C(K_0)$
 on $K_0$ is written as  
\[
\int_{K_0} f(k) dS(k).
\]
The product of $S, T\in R(K_0)$ is given by 
\[
S*T: f\mapsto \int_{K_0\times K_0} f(kk') dS(k)dT(k').
\]
The associative algebra $R(K_0)$
 does not have the identity, but has an approximate identity
 (see \cite[Chapter I]{KnVo}).
The locally finite $K_0$-modules are identified with the 
 approximately unital left $R(K_0)$-modules. 
The action map $R(K_0)\times V\to V$ is given by
\[
(S, v)\mapsto \int_{K_0} kv \,dS(k)
\]
for a locally finite $K_0$-module $V$. 
Here, $kv$ is regarded as a smooth function on $K_0$ that takes values on $V$.
If $dk_0$ denotes the Haar measure of $K_0$,
 then $R(K_0)$ is identified with the $K$-finite smooth functions
 $C(K_0)_{K_0}$ by $f dk_0\mapsto f$ and hence with
 the regular functions ${\cal O}(K)$ on $K$.
As a $\bb{C}$-algebra, we have a canonical isomorphism 
 \[R(K_0)\simeq \bigoplus_{\tau\in \widehat{K}} \operatorname{End}_\bb{C}
 (V_\tau),\] 
where $\widehat{K}$ is the set of equivalence classes of
 irreducible $K$-modules, and $V_\tau$ is a representation space
 of $\tau\in {\widehat{K}}$.
Hence $R(K_0)$ depends only on the complexification $K$, 
 so in what follows, we also denote $R(K_0)$ by $R(K)$.

The Hecke algebra $R(K)$ is generalized to $R(\frak{g}, K)$
 for the following pairs $(\frak{g},K)$.

\begin{de}
\label{de:pair}
Let $\frak{g}$ be a finite-dimensional complex Lie algebra
 and let $K$ be a complex reductive linear algebraic group with Lie algebra 
 $\frak{k}$.
Suppose that $\frak{k}$ is a Lie subalgebra of $\frak{g}$ and
 that an algebraic group homomorphism
 $\phi:K\to \operatorname{Aut}(\frak{g})$ is given.
We say that $(\frak{g},K)$ is a {\it pair}
 if the following two assumptions hold.
\begin{itemize}
\item
The restriction $\phi(k)|_{\frak{k}}$ is equal to
 the adjoint action $\Ad(k)$ for $k\in K$.
\item
The differential of $\phi$ is equal to
 the adjoint action $\ad_\frak{g}(\frak{k})$.
\end{itemize}
\end{de}

\begin{rem}
Let $G$ be a complex algebraic group and $K$ a reductive linear 
 algebraic subgroup.
Then the Lie algebra $\frak{g}$ of $G$ and $K$ form a pair 
 with respect to the adjoint action $\phi(k):={\Ad}(k)$ for $k\in K$.
All the pairs we will consider in this article are given in this way.
\end{rem}

\begin{de}
{\rm 
For a pair $(\frak{g},K)$,
 let $V$ be a complex vector space with a Lie algebra action of $\frak{g}$ and
 an algebraic action of $K$.
We say that $V$ is a {\it $(\frak{g},K)$-module} if 
\begin{itemize}
\item
the differential of the action of $K$ coincides with
 the restriction of the action of $\frak{g}$ to $\frak{k}$; and
\item
$(\phi(k)\xi)v=k(\xi (k^{-1}(v)))$ for $k\in K$,
 $\xi\in\frak{g}$, and $v\in V$.
\end{itemize}
}
\end{de}

We write ${\cal C}(\frak{g},K)$ for the category of
 $(\frak{g},K)$-modules.

Let $(\frak{g},K)$ be a pair in the sense of
 Definition~\ref{de:pair}.
We extend the representation $\phi: K\to \operatorname{Aut}(\frak{g})$ to
 a representation on the universal enveloping algebra
 $\phi: K\to \operatorname{Aut}(U(\frak{g}))$.
Define the Hecke algebra $R(\frak{g},K)$ as
 \[R(\frak{g},K):= R(K)\otimes_{U(\frak{k})} U(\frak{g}).\]
The product is given by
\[(S\otimes \xi)\cdot(T\otimes \eta)=
\sum_i (S * (\langle\xi_i^*, \phi(\cdot)^{-1}\xi\rangle T)\otimes \xi_i \eta)
\]
for $S, T\in R(K)$ and $\xi,\eta\in U(\frak{g})$.
Here $\xi_i$ is a basis of the linear span of $\phi(K)\xi$ and 
 $\xi^i$ is its dual basis.
As in the group case, 
 the $(\frak{g},K)$-modules are identified with the 
 approximately unital left $R(\frak{g}, K)$-modules. 
The action map $R(\frak{g},K)\times V\to V$ is
 given by
\[
(S\otimes \xi,\, v)\mapsto \int_{K_0} k(\xi v)\,dS(k)
\]
for a $(\frak{g}, K)$-module $V$.

Let $(\frak{g},K)$ and $(\frak{h},M)$ be pairs in the sense of
 Definition~\ref{de:pair}.
Let $i:(\frak{h},M)\to (\frak{g},K)$ be a map between pairs, namely,
 a Lie algebra homomorphism $i_{\rm alg}:\frak{h}\to \frak{g}$ and
 an algebraic group homomorphism $i_{\rm gp}:M\to K$ satisfy the following
 two assumptions.
\begin{itemize}
\item 
The restriction of $i_{\rm alg}$ to the Lie algebra $\frak{m}$ 
 of $M$ is equal to the differential of $i_{\rm gp}$.
\item 
$\phi_K(m)\circ i_{\rm alg}=i_{\rm alg}\circ\phi_M(m)$ for $m\in M$,
 where $\phi_K$ denotes $\phi$ for $(\frak{g},K)$ in
 Definition~\ref{de:pair} and $\phi_M$ denotes $\phi$ for $(\frak{h},M)$.
\end{itemize}
We define covariant functors
 $P_{\frak{h},M}^{\frak{g},K}
 :{\cal C}(\frak{h},M)\to {\cal C}(\frak{g},K)$
 and 
 $I_{\frak{h},M}^{\frak{g},K}
 :{\cal C}(\frak{h},M)\to {\cal C}(\frak{g},K)$
 as
\begin{align*}
&P_{\frak{h},M}^{\frak{g},K}
:V\mapsto R(\frak{g},K)\otimes_{R(\frak{h},M)} V, \\
&I_{\frak{h},M}^{\frak{g},K}
:V\mapsto ({\Hom}_{R(\frak{h},M)} (R(\frak{g},K),\,V))_{K},
\end{align*}
where $(\cdot)_{K}$ is the subspace of $K$-finite vectors.
Then $P_{\frak{h},M}^{\frak{g},K}$ is right exact
 and $I_{\frak{h},M}^{\frak{g},K}$ is left exact.
Write $(P_{\frak{h},M}^{\frak{g},K})_j$ for the $j$-th
 left derived functor of $P_{\frak{h},M}^{\frak{g},K}$ and
 write $(I_{\frak{h},M}^{\frak{g},K})^j$ for the $j$-th
 right derived functor of $I_{\frak{h},M}^{\frak{g},K}$.

In the context of unitary representations of
 real reductive Lie groups, we are especially interested in
 the $(\frak{g},K)$-modules cohomologically induced from
 one-dimensional representations of a certain type of
 parabolic subalgebras, which are called $A_\frak{q}(\lambda)$.

Let $G_0$ be a connected real linear reductive Lie group
 with Lie algebra $\frak{g}_0$.
This means that $G_0$ is a connected closed subgroup of $GL(n,\bb{R})$
 and stable under transpose.
We fix such an embedding and write $G$ for the connected 
 algebraic subgroup of $GL(n,\bb{C})$ with Lie algebra
 $\frak{g}=\frak{g}_0\oplus\sqrt{-1}\frak{g}_0$.
In what follows, we embed reductive subgroups
 of $G_0$ in $GL(n,\bb{C})$ and 
 define their complexifications similarly.

Fix a Cartan involution $\theta$
 so the $\theta$-fixed point set
 $K_0=G_0^{\theta}$ is a maximal compact
 subgroup of $G_0$.
Let $\frak{g}_0=\frak{k}_0+\frak{p}_0$ be the corresponding
 Cartan decomposition.
We let $\theta$ also denote the induced involution on $\frak{g}_0$
 and its complex linear extension to $\frak{g}$.

Let $\frak{q}$ be a parabolic subalgebra of $\frak{g}$
 that is stable under $\theta$.
The normalizer $N_{G_0}(\frak{q})$ of $\frak{q}$ in $G_0$
 is denoted by $L_0$.
The complexified Lie algebra $\frak{l}$ of $L_0$
 is a Levi part of $\frak{q}$.
Let bar $x \mapsto \Bar{x}$ denote the complex conjugate
 with respect to the real form $\frak{g}_0$.
Then we have $\frak{q}\cap\Bar{\frak{q}}=\frak{l}$ and
  $\frak{q}=\frak{l}+\frak{u}$
 for the nilradical $\frak{u}$ of $\frak{q}$.

Because $L\cap K$ is connected,
 one-dimensional $(\frak{l}, L\cap K)$-modules are
 determined by the action of the center $\frak{z}(\frak{l})$ of $\frak{l}$.
Let $\bb{C}_\lambda$ denote the one-dimensional
 $(\frak{l}, L\cap K)$-module
 corresponding to
 $\lambda\in \frak{z}(\frak{l})^*:=
 {\Hom}_\bb{C}(\frak{z}(\frak{l}),\bb{C})$.
With our normalization,
 the trivial representation corresponds to $\bb{C}_0$.
The top exterior product
 $\bigwedge^{\rm top}(\frak{g}/\Bar{\frak{q}})$ regarded as
 an $(\frak{l},L\cap K)$-module by the adjoint action
 corresponds to 
 $\bb{C}_{2\rho(\frak{u})}$
 for $2\rho(\frak{u}):={\rm Trace}\ad_{\frak{u}}(\cdot)$.

\begin{de}
{\rm 
Let $\bb{C}_\lambda$ be a one-dimensional
 $(\frak{l},L\cap K)$-module.

We say $\lambda$ is {\it unitary} if
 $\lambda$ takes pure imaginary
 values on the center $\frak{z}(\frak{l}_0)$ of $\frak{l}_0$,
 or equivalently, 
 if $\bb{C}_\lambda$ is the underlying $(\frak{l},L\cap K)$-module
 of a unitary character of $L_0$.

We say $\lambda$ is {\it linear} if $\bb{C}_\lambda$ lifts
 to an algebraic representation of the complexification $L$ of $L_0$.
}
\end{de}

\begin{rem}
\label{pzero}
{\rm 
If $\lambda$ is linear, then
 $\lambda$ takes real values on $\frak{z}(\frak{l}_0)\cap \frak{p}_0$.
In particular, if $\lambda$ is linear and unitary,
 then $\lambda$ is zero on $\frak{z}(\frak{l})\cap \frak{p}$.
}
\end{rem}

Let $\bb{C}_\lambda$ be
 a one-dimensional $(\frak{l},L\cap K)$-module.
We see 
 $\bb{C}_{\lambda+2\rho(\frak{u})}\simeq
 \bb{C}_\lambda\otimes \bb{C}_{2\rho(\frak{u})}$
 as a $(\Bar{\frak{q}},L\cap K)$-module
 (resp.\ a $(\frak{q},L\cap K)$-module)
 by letting $\Bar{\frak{u}}$ (resp.\ $\frak{u}$) acts as zero.
Then, for inclusion maps of pairs
 $(\Bar{\frak{q}},L\cap K)\to (\frak{g},K)$ and
 $(\frak{q},L\cap K)\to (\frak{g},K)$, define
 the cohomologically induced modules
 $(P_{\Bar{\frak{q}},L\cap K}^{\frak{g},K})_j
 (\bb{C}_{\lambda+2\rho(\frak{u})})$
 and
 $(I_{\frak{q},L\cap K}^{\frak{g},K})^j
 (\bb{C}_{\lambda+2\rho(\frak{u})})$.

The functor
 $P_{\frak{g},L\cap K}^{\frak{g},K}$ is called
 the Bernstein functor and denoted by $\Pi_{L\cap K}^{K}$.
Since
$P_{\Bar{\frak{q}},L\cap K}^{\frak{g},K}
 \simeq \Pi_{L\cap K}^{K}
 \circ P_{\Bar{\frak{q}},L\cap K}^{\frak{g},L\cap K}$
 and $P_{\Bar{\frak{q}},L\cap K}^{\frak{g},L\cap K}$
 is exact,
 it follows that $(P_{\Bar{\frak{q}},L\cap K}^{\frak{g},K})_j
 \simeq (\Pi_{L\cap K}^{K})_j
 \circ P_{\Bar{\frak{q}},L\cap K}^{\frak{g},L\cap K}$
 for the $j$-th left derived functor $(\Pi_{L\cap K}^K)_j$
 of $\Pi_{L\cap K}^K$.
Therefore, we have
\[
(P_{\Bar{\frak{q}},L\cap K}^{\frak{g},K})_j
 (\bb{C}_{\lambda+2\rho(\frak{u})})
\simeq (\Pi_{L\cap K}^{K})_j
 (U(\frak{g})\otimes_{U(\Bar{\frak{q}})}
 \bb{C}_{\lambda+2\rho(\frak{u})}).
\]
Similarly, 
 $\Gamma_{L\cap K}^{K}:=I_{\frak{g},L\cap K}^{\frak{g},K}$
 is called
 the Zuckerman functor and we have
\[
(I_{\frak{q},L\cap K}^{\frak{g},K})^j
 (\bb{C}_{\lambda+2\rho(\frak{u})})
\simeq (\Gamma_{L\cap K}^{K})^j
 ({\Hom}_{U(\frak{q})}
 (U(\frak{g}), \bb{C}_{\lambda+2\rho(\frak{u})})_{L\cap K})
\]
for the $j$-th right derived functor $(\Gamma_{L\cap K}^K)^j$
 of $\Gamma_{L\cap K}^K$.
Put $s=\dim (\frak{u}\cap \frak{k})$.
We define
\[
A_\frak{q}(\lambda):=
(P_{\Bar{\frak{q}},L\cap K}^{\frak{g},K})_s
(\bb{C}_{\lambda+2\rho(\frak{u})})
\simeq (\Pi_{L\cap K}^{K})_s 
(U(\frak{g})\otimes_{U(\Bar{\frak{q}})}
 \bb{C}_{\lambda+2\rho(\frak{u})}).
\]

We now discuss the positivity of the parameter $\lambda$.
Let $\frak{h}_0$ be a fundamental Cartan subalgebra of $\frak{l}_0$.
Choose a positive system $\Delta^+(\frak{g},\frak{h})$
 of the root system $\Delta(\frak{g},\frak{h})$
 such that $\Delta^+(\frak{g},\frak{h})
 \subset\Delta(\frak{q},\frak{h})$ and put 
\[\frak{n}=\bigoplus_{\alpha\in\Delta^+(\frak{g},\frak{h})}
\frak{g}_\alpha.\]
We fix a non-degenerate invariant form $\langle \cdot,\cdot\rangle$
 that is positive definite on the real span of the roots.
In the following definition,
 we extend characters of $\frak{z}(\frak{l})$ to
 $\frak{h}$ by zero on $[\frak{l},\frak{l}]\cap \frak{h}$.

\begin{de}
{\rm 
Let $\bb{C}_\lambda$ be a one-dimensional $(\frak{l},L\cap K)$-module.
We say $\lambda$ is in the {\it good range} 
(resp. {\it weakly good range}) if 
\begin{align*}
{\rm Re}\, \langle \lambda+\rho(\frak{n}), \alpha
 \rangle > 0\;{\rm (resp}.\geq 0{\rm )} 
\  {\rm for}\  \alpha\in \Delta(\frak{u},\frak{h}),
\end{align*}
and in the {\it fair range} 
(resp. {\it weakly fair range}) if 
\begin{align*}
{\rm Re}\, \langle \lambda+\rho(\frak{u}), \alpha
 \rangle > 0\;{\rm (resp}.\geq 0{\rm)}  \ 
{\rm for}\ \alpha\in \Delta(\frak{u},\frak{h}).
\end{align*}
}
\end{de}

\begin{de}
{\rm 
Let $V$ be a $(\frak{g},K)$-module.
We say $V$ is {\it unitarizable} 
if $V$ admits a Hermitian inner product with respect to which 
$\frak{g}_0$ acts by skew-Hermitian operators on $V$.
}
\end{de}

The $(\frak{g},K)$-module
 $A_\frak{q}(\lambda)$ has the following properties.

\begin{fact}[{\cite{KnVo}}]
\label{aqlam}
Let $\bb{C}_\lambda$ be a one-dimensional
 $(\frak{l},L\cap K)$-module.

\begin{itemize}
\item[$({\rm{i}})$] 
$A_\frak{q}(\lambda)$ is of finite length as a $(\frak{g},K)$-module.

\item[$({\rm{ii}})$] 
If $\lambda$ is in the weakly good range,
 $A_\frak{q}(\lambda)$ is irreducible or zero.

\item[$({\rm{iii}})$] 
If $\lambda$ is in the good range,
 $A_\frak{q}(\lambda)$ is nonzero.

\item[$({\rm{iv}})$] 
If $\lambda$ is unitary and in the weakly fair range, 
 then $A_\frak{q}(\lambda)$ is unitarizable.
\end{itemize}
\end{fact}


\section{Differential Operators on Homogeneous Spaces}
\label{sec:homog}
We introduce notation and lemmas concerning
 homogeneous spaces and differential operators, 
 used in the subsequent sections.
Let $G$ be a complex linear algebraic group
 acting on a smooth variety $X$.
Then the infinitesimal action is defined as a Lie algebra homomorphism
 from the Lie algebra $\frak{g}$ of $G$ to the space of
 vector fields ${\cal T}(X)$ on $X$.
Denote the image of $\xi\in\frak{g}$ by $\xi_X\in{\cal T}(X)$.
Then $\xi_X$ gives a first order differential operator
 on the structure sheaf ${\cal O}_X$.

Suppose that $X=G$ and the action of $G$ on $X$
 is the product from left: 
\[G\to \operatorname{Aut}(X),\qquad g\mapsto (g'\mapsto gg')\]
In this case we write the vector field $\xi_X$ as $\xi_G^L$,
 which is a right invariant vector field on $G$.
Similarly, if the action of $G$ on $X=G$ is the product from right:
 \[G\to \operatorname{Aut}(X),\qquad g\mapsto (g'\mapsto g'g^{-1}),\]
 we write the vector field $\xi_X$ as $\xi_G^R$, 
 which is a left invariant vector field on $G$.
Let $\xi_1,\cdots,\xi_n$ be a basis of $\frak{g}$
 and write $\xi^1,\cdots,\xi^n \in\frak{g}^*$ for the dual basis.
Define regular functions $\alpha^i_j, \beta^j_i$ on $G$
 for $1\leq i,j\leq n$ by
\begin{align}
\label{lrdef}
\alpha^i_j(g):=\langle\xi^i, \Ad(g^{-1})\xi_j \rangle,\quad
\beta^j_i(g):=\langle\xi^j, \Ad(g)\xi_i\rangle.
\end{align}
Then it follows that
\begin{align*}
(\xi_j)_G^L=-\sum_{i=1}^n \alpha^i_j\cdot(\xi_i)_G^R, \quad
(\xi_i)_G^R=-\sum_{j=1}^n \beta^j_i\cdot(\xi_j)_G^L, \quad
\sum_{j=1}^n \alpha^i_j\beta^j_k=\delta^i_k.
\end{align*}
We see $(\xi_j)_G^L$ as a differential operator on $G$. 
Then the function $(\xi_j)_G^L(\beta^j_i)$ on $G$ is written as 
\[(\xi_j)_G^L(\beta^j_i)=
-\langle \xi^j, [\xi_j, \Ad (\cdot) \xi_i]\rangle.\]
Hence
\begin{align}
\label{lrch}
\sum_{j=1}^n(\xi_j)_G^L(\beta^j_i)=
-\sum_{j=1}^n \langle \xi^j, [\xi_j, \Ad (\cdot) \xi_i]\rangle
={\rm Trace}\ad (\Ad(\cdot)\xi_i)={\rm Trace}\ad (\xi_i).
\end{align}

Let $H$ be a complex algebraic subgroup of $G$. 
The quotient $X:=G/H$ is defined as a smooth algebraic variety
 (see \cite[\S II.6]{Bo}). 
Denote by $\pi:G\to X$ the quotient map.
Let $V$ be a complex vector space with an algebraic action $\rho$ of $H$.
We define the
 {\it ${\cal O}_X$-module ${\cal V}_X$ associated with $V$}
 as the subsheaf of $\pi_*{\cal O}_G\otimes V$ given by
\[{\cal V}_X(U):=\{f\in {\cal O}(\pi^{-1}(U))\otimes V:
 f(gh)=\rho(h)^{-1}f(g)\}\]
for an open set $U\subset X$.
Here, we identify sections of ${\cal O}(\pi^{-1}(U))\otimes V$
 with regular $V$-valued functions on $\pi^{-1}(U)$.
Analogous identification will be used for other varieties.
The ${\cal O}_X$-module ${\cal V}_X$ corresponds to
 the $G$-equivariant vector bundle with typical fiber $V$.

The $G$-equivariant structure on ${\cal O}_G$ by the left translation induces 
 a $G$-equivariant structure on ${\cal V}_X$.
By differentiating it,
 the infinitesimal action of $\xi\in\frak{g}$ is given by
 $f\mapsto \xi_G^L f$.

We write $\operatorname{Ind}_H^{G}(V)$ for the
 space of global sections $\Gamma(X, {\cal V}_X)$
 regarded as an algebraic $G$-module.
Then by the Frobenius reciprocity,
\[{\Hom}_G(W, \operatorname{Ind}_H^{G}(V))
\xrightarrow{\sim} {\Hom}_H(W,V)
\]
for any algebraic $G$-module $W$.

\begin{lem}
\label{indred}
If $G$ and $H$ are reductive, then
\[
R(G)\otimes_{R(H)} V \simeq \operatorname{Ind}_H^{G}(V)
\]
as $G$-modules.
\end{lem}

\begin{proof}
We give the $H$-action on ${\cal O}(G)\otimes_\bb{C} V$ by
 $h(f\otimes v)\mapsto f(\,\cdot\,h)\otimes hv$.
The $H$-module ${\cal O}(G)\otimes_\bb{C} V$ decomposes as 
 a direct sum of irreducible
 factors because $H$ is reductive.
From the definition of ${\cal V}_X$, the space of global sections
 $\operatorname{Ind}_H^{G}(V)$ is equal to the set of $H$-invariant elements
 $({\cal O}(G)\otimes_\bb{C} V)^H$.
With the identification ${\cal O}(G)\simeq R(G)$, 
 we see that the canonical surjective map
 $R(G) \otimes_\bb{C} V \to R(G)\otimes_{R(H)} V$ is the projection onto
 the $H$-invariants.
Hence we have 
\[
R(G)\otimes_{R(H)} V\simeq ({\cal O}(G)\otimes_\bb{C} V)^H
 \simeq \operatorname{Ind}_H^{G}(V)
\]
as $G$-modules.
\end{proof}

Suppose that $H'$ is another algebraic subgroup
 of $G$ such that $H\subset H'$.
Let $X':=G/H'$ and $S:=H'/H$ be the quotient varieties and
 $\varpi : X\to X'$ the canonical map.
Write ${\cal V}_S$ for the ${\cal O}_S$-module associated with $V$.
Let $W:={\rm Ind}_H^{H'}(V)$ and
 let ${\cal W}_{X'}$ be the ${\cal O}_{X'}$-module
 associated with the $H'$-module $W$.

The following lemma is immediate from the definition,
 which indicates `induction by stages'
 in our setting.
\begin{lem}
\label{push}
In the setting above, there is a canonical 
 $G$-equivariant isomorphism 
$\varpi_*{\cal V}_X\to {\cal W}_{X'}.$
\end{lem}

Let $K$ be an algebraic subgroup of $G$. 
The inclusion map $i:K \to G$ induces the immersion
 $i: Y:=K /(H\cap K)\to X$ of algebraic variety.
Define the ideal ${\cal I}_Y$ of ${\cal O}_X$ as 
\[{\cal I}_Y:=\{ f\in {\cal O}_X : f(y)=0 {\rm\ for\ } y\in Y\},
\] 
so ${\cal I}_Y$ is the defining ideal of the closure $\ovl{Y}$ of $Y$.
We denote by ${\cal I}_Y^p$ the $p$-th power of ${\cal I}_Y$ for $p\geq 0$.
We use $i^{-1}$ for the inverse image of sheaves of abelian groups.
Then $i^{-1}({\cal I}_Y^p/{\cal I}_Y^{p+1})$
 is isomorphic to the $K$-equivariant ${\cal O}_Y$-module 
 associated with the dual of the $p$-th symmetric tensor product
 $S^p(\frak{g}/(\frak{h}+\frak{k}))^*$
 with the coadjoint action of $H \cap K$.
Let ${\cal T}_X$ be the sheaf of vector fields in $X$
 and let ${\cal T}_{X/Y}$ be the sheaf of vector fields in $X$
 tangent to $Y$, namely
\[
{\cal T}_{X/Y}
:= \{\xi\in {\cal T}_{X}:\xi({\cal I}_Y)\subset{\cal I}_Y\}.
\]
Then $\xi\in {\cal T}_X$ operates on ${\cal O}_X$ and
 induces an ${\cal O}_Y$-homomorphism
\[\xi: i^{-1}({\cal I}_Y/{\cal I}_Y^2)\to 
i^{-1}({\cal O}_X/{\cal I}_Y)\simeq {\cal O}_Y.\]
This gives an isomorphism of locally free ${\cal O}_Y$-modules 
\[
i^{-1}({\cal T}_X/{\cal T}_{X/Y})
\simeq  {\cal H}om_{{\cal O}_Y}
(i^{-1}({\cal I}_Y/{\cal I}_Y^2), {\cal O}_Y),
\]
which correspond to the normal bundle of $Y$ in $X$.

We denote by ${\cal D}_X$ the ring of differential operators
 on $X$. Then ${\cal D}_X$ has the filtration given by 
\[F_p {\cal D}_X:=
\{D\in {\cal D}_X : 
\xi ({\cal I}_Y^{p+1}) \subset {\cal I}_Y\}, 
\]
 which is called the filtration by normal degree with respect to $i$.
A section of $F_p {\cal D}_X$ is locally written as
 $\sum \eta_1\cdots \eta_r\xi_1 \dots \xi_q$,
 where $q\leq p$, $\xi_1,\dots,\xi_q\in {\cal T}_X$,
 and $\eta_1,\dots,\eta_r\in {\cal T}_{X/Y}$.
Let $G_p {\cal D}_X(\subset{\cal D}_X)$ be the sheaf of
 differential operators on $X$ with rank equal or less than $p$.
For $D\in G_p {\cal D}_X$, the
 differential operator $D:{\cal O}_X\to {\cal O}_X$ induces
 an ${\cal O}_Y$-homomorphism
\[i^{-1}({\cal I}_Y^p/{\cal I}_Y^{p+1})
 \to i^{-1}({\cal O}_X/{\cal I}_Y)
 \simeq {\cal O}_Y, 
\]
which we denote by $\gamma(D)$.
Write 
\[i^{-1}({\cal I}_Y^p/{\cal I}_Y^{p+1})\spcheck
:=  {\cal H}om_{{\cal O}_Y}
 (i^{-1}({\cal I}_Y^p/{\cal I}_Y^{p+1}),{\cal O}_Y)
\]
for the dual of $i^{-1}({\cal I}_Y^p/{\cal I}_Y^{p+1})$.
The map $D\mapsto \gamma(D)$ gives an isomorphism
 of ${\cal O}_Y$-modules
\begin{align}
\label{eqn:dfil}
i^{-1}G_p {\cal D}_X
/i^{-1}(G_p {\cal D}_X\cap F_{p-1}{\cal D}_X)
&\simeq i^{-1}({\cal I}_Y^p/{\cal I}_Y^{p+1})\spcheck.
\end{align}
They are also isomorphic to the $p$-th symmetric tensor of
 the locally free ${\cal O}_Y$-module
 $i^{-1}({\cal I}_Y/{\cal I}_Y^2)\spcheck$.

Let ${\cal M}$ be a left ${\cal D}_Y$-module.
The Lie algebra $\frak{k}$ acts on ${\cal M}$ by
 $\eta_Y$ for $\eta\in\frak{k}$.
Write $\Omega_X$ and $\Omega_Y$ for the canonical sheaves of $X$ and $Y$,
 respectively.
The push-forward by $i$ is defined by 
\[i_+{\cal M}:=i_*(({\cal M}\otimes_{{\cal O}_Y} \Omega_Y)
 \otimes_{{\cal D}_Y}
 {i^*{\cal D}_X})
 \otimes_{{\cal O}_X} \Omega_X\spcheck. \]
Here, we write $i_*$ for the push-forward of ${\cal O}$-modules or
 $\bb{C}$-modules and
 $i_+$ for the push-forward of ${\cal D}$-modules.
$i^*$ denotes the pull-back of ${\cal O}$-modules.
It follows from the definition that 
\[i^{-1}i_+{\cal M}\simeq
 ({\cal M}\otimes_{{\cal O}_Y} \Omega_Y)
 \otimes_{{\cal D}_Y}
 ({{\cal O}_Y\otimes_{i^{-1}{\cal O}_X}i^{-1}{\cal D}_X})
 \otimes_{i^{-1}{\cal O}_X} i^{-1}\Omega_X\spcheck. \]
By using the filtration by normal degree, 
 we define the $(i^{-1}{\cal O}_X)$-module 
\[F_p i^{-1}i_+{\cal M}:=({\cal M}\otimes_{{\cal O}_Y} \Omega_Y)
 \otimes_{{\cal D}_Y}
 ({\cal O}_Y \otimes_{i^{-1}{\cal O}_X}i^{-1}F_p {\cal D}_X)
 \otimes_{i^{-1}{\cal O}_X} i^{-1}\Omega_X \spcheck 
\]
for $p\geq 0$.
This is well-defined because
 ${\cal O}_Y \otimes_{i^{-1}{\cal O}_X}i^{-1}F_p {\cal D}_X$ is
 stable under the left ${\cal D}_Y$-action.
We see that
 $i^{-1}F_p {\cal D}_X$ is a flat $(i^{-1}{\cal O}_X)$-module,
 ${\cal O}_Y \otimes_{i^{-1}{\cal O}_X}i^{-1}F_p {\cal D}_X$ is
 a left flat ${\cal D}_Y$-module,
 and $i^{-1}\Omega_X\spcheck$ is a flat $(i^{-1}{\cal O}_X)$-module.
Hence the $(i^{-1}{\cal O}_X)$-modules $F_p i^{-1}i_+{\cal M}$
 form a filtration of $i^{-1}i_+{\cal M}$.

Consider the restriction of the $\frak{g}$-action on $i_+{\cal M}$
 to $\frak{k}$.
For $\eta\in\frak{k}$, the vector field $\eta_X$ is tangent to $Y$.
Hence the $\frak{k}$-action stabilizes each $F_pi^{-1}i_+{\cal M}$ and
 it induces an action on the quotient
 $F_pi^{-1}i_+{\cal M}/F_{p-1}i^{-1}i_+{\cal M}$. 
Moreover, $F_p{\cal D}_X \cdot {\cal I}_Y\subset F_{p-1}{\cal D}_X$
 implies that
 $i^{-1}{\cal I}_Y \cdot F_pi^{-1}i_+{\cal M}\subset F_{p-1}i^{-1}i_+{\cal M}$.
Therefore $F_pi^{-1}i_+{\cal M}/F_{p-1}i^{-1}i_+{\cal M}$
 carries an ${\cal O}_Y$-module structure.
Write $\Omega_{X/Y}:= \Omega_Y \spcheck \otimes_{i^{-1}{\cal O}_X}
 i^{-1}\Omega_{X}$ for the relative canonical sheaf.
The $K$-equivariant structures on
 the ${\cal O}_Y$-modules $\Omega_{X/Y} \spcheck$ and
 $i^{-1}({\cal I}^p/{\cal I}^{p+1})$ give
 $\frak{k}$-actions on them.

\begin{lem}
\label{fil}
There is an isomorphism of ${\cal O}_Y$-modules
\[
F_p i^{-1}i_+{\cal M} / F_{p-1} i^{-1}i_+{\cal M}
\simeq {\cal M}\otimes_{{\cal O}_Y} \Omega_{X/Y} \spcheck
\otimes_{{\cal O}_Y} i^{-1}({\cal I}_Y^p/{\cal I}_Y^{p+1}) \spcheck
\]
 that commutes with the actions of $\frak{k}$.
Here, the $\frak{k}$-action on the right side is given by the
 tensor product of the action on each factors defined above.
\end{lem}

\begin{proof}
The inverse image
 $i^*{\cal D}_X:={\cal O}_Y\otimes_{i^{-1}{\cal O}_X}i^{-1}{\cal D}_X$
 of ${\cal D}_X$ in the category of ${\cal O}$-modules has a
 left ${\cal D}_Y$-module structure.
The action map
\[{\cal D}_Y \otimes_{{\cal O}_Y} 
 ({\cal O}_Y \otimes_{i^{-1}{\cal O}_X} i^{-1}{\cal D}_X)
\to {\cal O}_Y \otimes_{i^{-1}{\cal O}_X} i^{-1}{\cal D}_X
\]
induces a morphism of left ${\cal D}_Y$-modules
\begin{align}
\label{eqn:dfil2}
&{\cal D}_Y\otimes_{{\cal O}_Y}({\cal O}_Y\otimes_{i^{-1}{\cal O}_X} 
 i^{-1}(G_p{\cal D}_X/
 (G_p{\cal D}_X \cap F_{p-1}{\cal D}_X))) \\ \nonumber
\to\  &{\cal O}_Y\otimes_{i^{-1}{\cal O}_X}
 i^{-1}(F_p{\cal D}_X/F_{p-1}{\cal D}_X).
\end{align}
We give the inverse map of \eqref{eqn:dfil2}.
Any section of $F_p{\cal D}_X/ F_{p-1}{\cal D}_X$
 is represented by a sum of section of the form
 $\eta_1\cdots\eta_r\xi_1\cdots\xi_p$
 for $\xi_1,\dots,\xi_p \in {\cal T}_X$ and 
 $\eta_1,\dots,\eta_r\in {\cal T}_{X/Y}$.
The inverse map
\begin{align*}
&{\cal O}_Y\otimes_{i^{-1}{\cal O}_X}
 i^{-1}(F_p{\cal D}_X/F_{p-1}{\cal D}_X) \\
\to\  &{\cal D}_Y\otimes_{{\cal O}_Y}({\cal O}_Y\otimes_{i^{-1}{\cal O}_X} 
 i^{-1}(G_p{\cal D}_X/
 (G_p{\cal D}_X \cap F_{p-1}{\cal D}_X)))
\end{align*}
is given by
\begin{align*}
f\otimes \eta_1\cdots\eta_r\xi_1\cdots\xi_p
\mapsto f(\eta_1)|_Y\cdots(\eta_r)|_Y\otimes
 (1\otimes\xi_1\cdots\xi_p).
\end{align*}
Hence \eqref{eqn:dfil2} is an isomorphism.

By using \eqref{eqn:dfil} and \eqref{eqn:dfil2}, 
we obtain isomorphisms of ${\cal O}_Y$-modules: 
\begin{align}
\label{eqn:filisom}
&F_p i^{-1}i_+{\cal M} / F_{p-1} i^{-1}i_+{\cal M}\\ \nonumber
\simeq \ 
&({\cal M}\otimes_{{\cal O}_Y} \Omega_Y)
 \otimes_{{\cal D}_Y}
({\cal O}_Y\otimes_{i^{-1}{\cal O}_X}
 i^{-1}(F_p{\cal D}_X / F_{p-1} {\cal D}_X))
 \otimes_{i^{-1}{\cal O}_X} i^{-1}\Omega_X \spcheck \\ \nonumber
\simeq \ 
&({\cal M}\otimes_{{\cal O}_Y} \Omega_Y)
  \!\otimes_{{\cal D}_Y} \!
 ({\cal D}_Y \!\otimes_{{\cal O}_Y} \!
 ({\cal O}_Y  \!\otimes_{i^{-1}{\cal O}_X} \! 
 i^{-1}(G_p{\cal D}_X/
 (G_p{\cal D}_X \cap F_{p-1}{\cal D}_X)))) \\ \nonumber
&\hspace*{250pt}
 \otimes_{i^{-1}{\cal O}_X} i^{-1}\Omega_X \spcheck \\ \nonumber
\simeq \ 
&({\cal M}\otimes_{{\cal O}_Y} \Omega_Y)
 \otimes_{{\cal O}_Y}
 i^{-1}(G_p{\cal D}_X /
 (G_p{\cal D}_X\cap F_{p-1} {\cal D}_X))
 \otimes_{i^{-1}{\cal O}_X} i^{-1}\Omega_X \spcheck \\ \nonumber
\simeq \ 
& {\cal M}\otimes_{{\cal O}_Y} 
 \Omega_{X/Y} \spcheck
 \otimes_{{\cal O}_Y}
 i^{-1}({\cal I}_Y^p/{\cal I}_Y^{p+1})\spcheck.
\end{align}

We now show that this map commutes with the $\frak{k}$-actions.
Take a section
\[(m\otimes \omega)
 \otimes (1\otimes D)\otimes \omega'
 \in ({\cal M}\otimes_{\cal O} \Omega_Y) \otimes_{\cal D}
 ({\cal O}_Y\otimes_{i^{-1}{\cal O}}i^{-1}F_p{\cal D}_X)
 \otimes_{i^{-1}{\cal O}} i^{-1}\Omega_X \spcheck
\]
for $m\in {\cal M}$, $\omega\in\Omega_Y$,
 $D\in G_p{\cal D}_X$, and $\omega'\in \Omega_X \spcheck$.
Since any section of
 $F_pi^{-1}i_+{\cal M}/F_{p-1}i^{-1}i_+{\cal M}$
 is represented by a sum of 
 sections of this form, 
 it is enough to see the commutativity for this section.
Under the isomorphisms \eqref{eqn:filisom},
 the section $(m\otimes \omega) \otimes (1\otimes D)\otimes \omega'$ 
 corresponds to
 $m\otimes(\omega\otimes\omega')\otimes \gamma(D)
 \in{\cal M}\otimes_{\cal O} \Omega_{X/Y} \spcheck
 \otimes_{\cal O} i^{-1}({\cal I}^p/{\cal I}^{p+1})\spcheck$.
For $\eta\in\frak{k}$, the $\frak{k}$-action on
 $i^{-1}i_+{\cal M}$ is given by 
\begin{align*}
&(m\otimes \omega)
 \otimes (1\otimes D)\otimes \omega'\\
\mapsto\ &(m\otimes \omega)
 \otimes (1\otimes D(-\eta_X))\otimes \omega'
+(m\otimes \omega)
 \otimes (1\otimes D)\otimes \eta_X \omega' \\
=\ &(m\otimes \omega)
 \otimes (1\otimes (-\eta_X)D)\otimes \omega'
+(m\otimes \omega)
 \otimes (1\otimes [\eta_X, D])\otimes \omega'\\
&\qquad\qquad\qquad\qquad\qquad\qquad\quad
 +(m\otimes \omega)
 \otimes (1\otimes D)\otimes \eta_X \omega'.
\end{align*}
Since $\eta_X|_Y=\eta_Y$, it follows that
\begin{align*}
(m&\otimes \omega)
 \otimes (1\otimes (-\eta_X)D)\otimes \omega'
=(m\otimes \omega)(-\eta_Y)
 \otimes (1\otimes D)\otimes \omega' \\
&=(\eta_Y m\otimes \omega)
 \otimes (1\otimes D)\otimes \omega'
 +(m\otimes \omega(-\eta_Y))
 \otimes (1\otimes D)\otimes \omega'.
\end{align*}
As a result, the action of $\eta$ is given by
\begin{align*}
&\eta\cdot((m\otimes \omega)
 \otimes (1\otimes D)\otimes \omega')\\
=\ &(\eta_Y m\otimes \omega)
 \otimes (1\otimes D)\otimes \omega'
 +(m\otimes \omega(-\eta_Y))
 \otimes (1\otimes D)\otimes \omega'\\
+\ &(m\otimes \omega)
 \otimes (1\otimes [\eta_X, D])\otimes \omega'
 +(m\otimes \omega)
 \otimes (1\otimes D)\otimes \eta_X \omega'.
\end{align*}
Since
 $[\eta_X, D]\in G_p{\cal D}_X$,
 the section $\eta\cdot((m\otimes \omega)
 \otimes (1\otimes D)\otimes \omega')$
 corresponds to 
\begin{align*}
\eta_Y m\otimes (\omega\otimes\omega')\otimes \gamma(D)
 &+m\otimes \eta_Y(\omega\otimes\omega')\otimes \gamma(D)\\
 &+m\otimes (\omega\otimes\omega')
    \otimes \gamma([\eta_X, D]).
\end{align*}
Thus, the commutativity follows from
 $\gamma([\eta_X, D])=\eta\cdot\gamma(D)$.
\end{proof}

In the rest of this section,
 we assume that $K$ and $H\cap K$ are 
 complex reductive linear algebraic groups. 
In particular, $Y:=K/(H\cap K)$ is an affine variety 
 by \cite[\S I.2]{Mum}.

We assume moreover that there exists a $K$-equivariant isomorphism of
 ${\cal O}_Y$-modules: $\Omega_Y\simeq{\cal O}_Y$, or equivalently, 
 the $(H\cap K)$-module $\bigwedge^{\rm top}(\frak{h}/(\frak{h}\cap\frak{k}))$
 with the adjoint action is trivial.
This assumption automatically holds if $H\cap K$ is connected.

Let $V$ be an $H$-module. 
Then $V$ is written as a union of finite-dimensional $H$-submodules 
 and has a structure of $(\frak{h}, H\cap K)$-module.
Define the $(\frak{g},K)$-module
 $R(\frak{g},K)\otimes_{R(\frak{h}, H\cap K)} V$
 as in Section~\ref{sec:coh}.

Let ${\cal V}_X$ be the ${\cal O}_X$-module associated with
 the $H$-module $V$.
Then the $G$-equivariant structures of ${\cal V}_X$ and ${\Omega}_X$
 induce $(\frak{g},K)$-actions on them.

The next lemma relates these two modules.

\begin{lem}
\label{affloc}
Under the assumptions above, 
there is an isomorphism of $(\frak{g},K)$-modules
\[
R(\frak{g},K)\otimes_{R(\frak{h},H\cap K)} V
\xrightarrow{\sim} 
\Gamma(X, i_+{\cal O}_Y \otimes_{{\cal O}_X} \Omega_X 
 \otimes_{{\cal O}_X} {\cal V}_X),
\] 
where the actions of $\frak{g}$ and $K$ on the right side are given by
 the tensor product of three factors.
\end{lem}

\begin{proof}
With the identification $\Omega_Y\simeq {\cal O}_Y$, 
we have 
\[
i_+{\cal O}_Y \otimes_{{\cal O}_X} \Omega_X
\simeq i_*({\cal O}_Y \otimes_{{\cal D}_Y} i^*{\cal D}_X). 
\]
Hence
\[ i^{-1}(i_+{\cal O}_Y \otimes_{{\cal O}_X} \Omega_X 
 \otimes_{{\cal O}_X} {\cal V}_X)
\simeq {\cal O}_Y\otimes_{{\cal D}_Y} 
(i^*{\cal D}_X \otimes_{i^{-1}{\cal O}_X} i^{-1}{\cal V}_X).\]
Using the right $(i^{-1}{\cal D}_X)$-module structure of
 $i^*{\cal D}_X$, we define a $\frak{g}$-action $\rho$ on
 the sheaf $i^*{\cal D}_X \otimes_{i^{-1}{\cal O}_X} i^{-1}{\cal V}_X$ by 
\[
\rho(\xi)(D\otimes v):= D(-\xi_X)\otimes v+D\otimes \xi v
\]
for $\xi\in\frak{g}$, $D\in i^*{\cal D}_X$, and
 $v\in {\cal V}_X$.
Moreover,
 the sheaf $i^*{\cal D}_X \otimes_{i^{-1}{\cal O}_X} i^{-1}{\cal V}_X$
 is $K$-equivariant.
We denote this $K$-action and also its infinitesimal
 $\frak{k}$-action by $\nu$.
Using the $({\cal D}_Y, i^{-1}{\cal D}_X)$-bimodule structure on
 $i^*{\cal D}_X$, the $\frak{k}$-action $\nu$ is given by
\[
\nu(\eta)(D\otimes v)= \eta_Y D \otimes v - D\eta_X \otimes v
+D\otimes \eta v
\]
for $\eta\in\frak{k}$.
Then $\Gamma(Y, i^*{\cal D}_X \otimes_{i^{-1}{\cal O}_X} i^{-1}{\cal V}_X)$
 is a weak Harish-Chandra module in the sense of \cite{MiPa}, namely,
\begin{align}
\label{weakHC}
\nu(k)\rho(\xi)\nu(k^{-1})=\rho(\Ad (k)\xi)
\end{align}
 for $k\in K$ and $\xi\in\frak{g}$.
Put $\omega(\eta):=\nu(\eta)-\rho(\eta)$
 for $\eta\in\frak{k}$.
Then $\omega(\eta)$ is given by
\[
\omega(\eta)(D\otimes v)= \eta_Y D\otimes v.
\]
Since $Y$ is an affine variety, 
 $\Gamma(Y, {\cal D}_Y)$ is generated by $U(\frak{k})$
 as an ${\cal O}(Y)$-algebra.
Therefore,
\begin{align*}
&\Gamma(X, i_+{\cal O}_Y\otimes_{{\cal O}_X} 
 \Omega_X \otimes_{{\cal O}_X} {\cal V}_X) \\
\simeq\ &{\cal O}(Y)\otimes_{\Gamma(Y, {\cal D}_Y)}
 \Gamma(Y, i^*{\cal D}_X \otimes_{i^{-1}{\cal O}_X} i^{-1}{\cal V}_X)\\
\simeq\ &\Gamma(Y, i^*{\cal D}_X \otimes_{i^{-1}{\cal O}_X}
  i^{-1}{\cal V}_X)
 \,/\,\omega(\frak{k}) \Gamma(Y, i^*{\cal D}_X \otimes_{i^{-1}{\cal O}_X} 
 i^{-1}{\cal V}_X).
\end{align*}

Let $e\in K$ be the identity element. 
Write $o:=e(H\cap K)\in Y$ for the base point
 and $i_{o,Y}:\{o\}\to Y$ for the immersion.
Let ${\cal I}_o$ be the maximal ideal of ${\cal O}_Y$
 corresponding to $o$.
The geometric fiber of
 $i^*{\cal D}_X \otimes_{i^{-1}{\cal O}_X} i^{-1}{\cal V}_X$
 at $o$ is given by
\begin{align*}
W:=\ &(i_{o,Y})^*(i^*{\cal D}_X \otimes_{i^{-1}{\cal O}_X} i^{-1}{\cal V}_X)\\
 \simeq\ &\Gamma(Y, i^*{\cal D}_X \otimes_{i^{-1}{\cal O}_X} i^{-1}{\cal V}_X)
 \,/\, {\cal I}_o(Y)
 \Gamma(Y, i^*{\cal D}_X \otimes_{i^{-1}{\cal O}_X} i^{-1}{\cal V}_X).
\end{align*}
The actions $\rho$ and $\nu$ on
 $i^*{\cal D}_X \otimes_{i^{-1}{\cal O}_X} i^{-1}{\cal V}_X$
 induce a $\frak{g}$-action $\rho_o$ and an
 $(H\cap K)$-action $\nu_o$ on $W$.
With these actions, $W$ becomes a $(\frak{g},H\cap K)$-module.
To show this, it is enough to see that $\rho_o$ and $\nu_o$
 agree on $\frak{h}\cap \frak{k}$.
This follows from
\[\omega(\eta)
 \Gamma(Y, i^*{\cal D}_X \otimes_{i^{-1}{\cal O}_X} i^{-1}{\cal V}_X)
 \subset {\cal I}_o(Y)
 \Gamma(Y, i^*{\cal D}_X \otimes_{i^{-1}{\cal O}_X} i^{-1}{\cal V}_X)\]
for $\eta\in\frak{h}\cap \frak{k}$.

We claim that $W\simeq U(\frak{g})\otimes_{U(\frak{h})} V$ as
 a $(\frak{g}, H\cap K)$-module.
Put $i_{o,X}:=i\circ i_{o,Y}$.
Then 
\begin{align*}
W &\simeq (i_{o,X})^*{\cal D}_X \otimes_{(i_{o,X})^{-1}{\cal O}_X}
 (i_{o,X})^{-1}{\cal V}_X \\
&\simeq (i_{o,X})^{-1}((i_{o,X})_+
 {\cal O}_{\{o\}}\otimes_{{\cal O}_X} \Omega_X)
  \otimes_{(i_{o,X})^{-1}{\cal O}_X}
 (i_{o,X})^{-1}{\cal V}_X.
\end{align*}
Let $\{F_p {\cal D}_X\}$ be the filtration by
 normal degree with respect to $i_{o,X}$.
Define the filtration
\[
F_p W:= (i_{o,X})^*F_p{\cal D}_X \otimes_{(i_{o,X})^{-1}{\cal O}_X}
 (i_{o,X})^{-1}{\cal V}_X
\]
 of $W$.
Then $F_p W$ is $(\frak{h},H\cap K)$-stable and
 there is an isomorphism of $(\frak{h},H\cap K)$-modules
\[F_p W/F_{p-1} W \simeq
 (i_{o,X})^{-1}({\cal I}_{o}^p/{\cal I}_{o}^{p+1})\spcheck 
 \otimes V
\]
 by Lemma~\ref{fil}.
The isomorphism $F_0 W\simeq V$ induces
 a $(\frak{g},H\cap K)$-homomorphism
 $\varphi: U(\frak{g})\otimes_{U(\frak{h})} V \to W$.
Let $U_p(\frak{g})$ be the standard filtration of $U(\frak{g})$.
Then 
$(U_p(\frak{g})U(\frak{h}))\otimes_{U(\frak{h})}V$
 is a filtration of the $(\frak{h},H\cap K)$-module
 $U(\frak{g})\otimes_{U(\frak{h})} V$ 
and there is an isomorphism of $(\frak{h},H\cap K)$-modules:
\[
(U_p(\frak{g})U(\frak{h}))\otimes_{U(\frak{h})}V \,/\,
(U_{p-1}(\frak{g})U(\frak{h}))\otimes_{U(\frak{h})}V
\simeq S^p(\frak{g}/\frak{h})\otimes V.
\]
In view of the proof of Lemma~\ref{fil},
 we see that the map on the successive quotient 
\[
\varphi_p: (U_p(\frak{g})U(\frak{h}))\otimes_{U(\frak{h})}V \,/\,
(U_{p-1}(\frak{g})U(\frak{h}))\otimes_{U(\frak{h})}V
\to F_p W/F_{p-1} W
\]
 induced by $\varphi$ is an isomorphism.
Hence $\varphi$ is an isomorphism.

As a $K$-equivariant ${\cal O}_Y$-module,
 $i^*{\cal D}_Y\otimes_{i^{-1}{\cal O}_X} i^{-1}{\cal V}_X$ is isomorphic to
 the ${\cal O}_Y$-module ${\cal W}_Y$ associated with 
 the $(H\cap K)$-module $W$.
Hence we can see global sections 
 $\Gamma(Y, i^*{\cal D}_Y\otimes_{i^{-1}{\cal O}_X} i^{-1}{\cal V}_X)$
 as $W$-valued regular functions on $K$.
Let $f$ be a $W$-valued regular function on $K$ such that
 $f(kh)=\nu_o(h^{-1}) f(k)$ for $k\in K$ and $h\in H\cap K$.
The $\frak{g}$-action $\rho$ at $e$ is given by
 $(\rho(\xi)f)(e)=\rho_o(\xi)(f(e))$.
Hence \eqref{weakHC} implies that
\begin{align*}
(\rho(\xi)f)(k)= (\nu(k)\rho(\Ad(k^{-1}) \xi)\nu(k^{-1})f)(k)
=\rho_o(\Ad(k^{-1})\xi)(f(k)).
\end{align*}
Let $\xi_1,\dots,\xi_n$ be a basis of $\frak{g}$
 and write $\xi^1,\dots,\xi^n\in\frak{g}^*$ for its dual basis.
Under the isomorphism
 $\Gamma(Y, {\cal W}_Y)\simeq R(K)\otimes_{R(H\cap K)} W$
 given in Lemma~\ref{indred}, 
the $\frak{g}$-action $\rho$ on $R(K)\otimes_{R(H\cap K)} W$ is given by
\begin{align}
\label{rhoact}
\rho(\xi)(S\otimes w)=
 \sum_{i=1}^n \langle \xi^i, \Ad(\cdot)^{-1}\xi\rangle S\otimes
 \rho_o(\xi_i) w
\end{align}
for $S\in R(K)$ and $w\in W$.
If we define $\rho$ on $R(K)\otimes_\bb{C} W$ by this equation, 
 then $\rho$ commutes with the canonical surjective map
\[p:R(K)\otimes_\bb{C} W \to R(K) \otimes_{R(H\cap K)} W.\]
The $K$-action $\nu$ is given by the left translation of $R(K)$:
\[
\nu(k)(S\otimes w)= (kS)\otimes w.
\]
Hence $\nu$ also lifts to the action on $R(K)\otimes_\bb{C} W$
and commutes with $p$.
Let $\eta_1,\cdots,\eta_m$ be a basis of $\frak{k}$ and
 write $\eta^1,\cdots,\eta^m\in\frak{k}^*$ for its dual basis.
Define the regular functions $\alpha^i_j$ and $\beta^j_i$
 on $K$ with respect to $\eta_i$ as in \eqref{lrdef}.
Then the $\frak{k}$-action $\omega$ is given by
\begin{align*}
\omega(\eta_j)(S\otimes w)&=
\nu(\eta_j)(S\otimes w)-\rho(\eta_j)(S\otimes w)\\
&=((\eta_j)_K^L S)\otimes w
-\sum_{i=1}^m \alpha^i_j S\otimes \rho_o(\eta_i) w.
\end{align*}
Here, we identify $R(K)$ with ${\cal O}(K)$, and
 give actions of differential operators on $K$.

We have 
\begin{align*}
&\Gamma(X, i_+{\cal O}_Y \otimes_{{\cal O}_X} \Omega_X
 \otimes_{{\cal O}_X} {\cal V}_X)\\
\simeq\ &\Gamma(Y, i^*{\cal D}_X \otimes_{i^{-1}{\cal O}_X}
 i^{-1}{\cal V}_X)
 /\omega(\frak{k})\Gamma(Y, i^*{\cal D}_X \otimes_{i^{-1}{\cal O}_X}
 i^{-1}{\cal V}_X)\\
\simeq\ &(R(K)\otimes_{R(H\cap K)} W)
 /\omega(\frak{k}) (R(K)\otimes_{R(H\cap K)} W).
\end{align*}
We note that the $\frak{k}$-actions $\rho$ and $\nu$ 
 agree on the quotient
$(R(K)\otimes_{R(H\cap K)} W)
 /\omega(\frak{k}) (R(K)\otimes_{R(H\cap K)} W)$
 and hence it becomes a $(\frak{g},K)$-module.

The equation
 $\sum_{j=1}^m \alpha^i_j\beta^j_k=\delta^i_k$ implies that
 $\omega(\frak{k}) (R(K)\otimes_{\bb{C}} W)$
 is generated by the elements of the form
$\sum_{j=1}^m \omega(\eta_j)(\beta^j_k S\otimes w)$
 for $S\in R(K)$ and $w\in W$.
We observe from \eqref{lrch} that 
$
\sum_{j=1}^m (\eta_j)_K^L(\beta^j_k)=0
$
because ${\rm Trace} \ad(\cdot) =0$ for the reductive Lie algebra $\frak{k}$.
Therefore, 
\[
(\eta_k)_K^R
=-\sum_{j=1}^m \beta^j_k(\eta_j)_K^L 
=-\sum_{j=1}^m (\eta_j)_K^L \beta^j_k
\]
as differential operators on $K$.
Then
\begin{align*}
\sum_{j=1}^m \omega(\eta_j)(\beta^j_k S\otimes w)
&= \sum_{j=1}^m (\eta_j)_K^L \beta^j_k S\otimes w
+ \sum_{i,j=1}^m (\alpha^i_j\beta^j_k  S\otimes \rho_o(\eta_i)w) \\
&= -(\eta_k)_K^R S\otimes w
+ S\otimes \rho_o(\eta_k)w.
\end{align*}

Consequently, 
 the kernel of the map
\begin{align*}
R(K)\otimes_\bb{C} W\to (R(K)\otimes_{R(H\cap K)} W)
/\omega(\frak{k}) (R(K)\otimes_{R(H\cap K)} W)
\end{align*}
 is generated by ${\rm Ker}\ p$ and 
 $-(\eta_k)_K^R S\otimes w
 + S\otimes \rho_o(\eta_k)w$.
Hence 
\begin{align*}
&(R(K)\otimes_{R(H\cap K)} W)
/\omega(\frak{k}) (R(K)\otimes_{R(H\cap K)} W) \\
\simeq\ &R(K)\otimes_{R(\frak{k}, H\cap K)} W \\
\simeq\ &R(\frak{g},K)\otimes_{R(\frak{g},H\cap K)} W.
\end{align*}
From \eqref{rhoact}, we see that the isomorphism
\[
(R(K)\otimes_{R(H\cap K)} W)
/\omega(\frak{k}) (R(K)\otimes_{R(H\cap K)} W)
\simeq R(\frak{g},K)\otimes_{R(\frak{g},H\cap K)} W
\]
 commutes with the $(\frak{g},K)$-actions.
Therefore, 
\begin{align*}
\Gamma(X, i_+{\cal O}_Y\otimes_{{\cal O}_X} \Omega_X
 \otimes_{{\cal O}_X} {\cal V}_X)
&\simeq R(\frak{g},K)\otimes_{R(\frak{g},H\cap K)} W\\
&\simeq R(\frak{g},K)\otimes_{R(\frak{h},H\cap K)} V
\end{align*}
and the lemma is proved.
\end{proof}


\section{Localization of the Cohomological Induction}
\label{sec:loc}

In this section, we construct cohomologically
 induced modules on flag varieties.
Let $G_0$ be a connected real linear reductive Lie group
 with Lie algebra $\frak{g}_0$
 and $\frak{q}$ a $\theta$-stable parabolic subalgebra as in
 Section \ref{sec:coh}.
We define the complexification $G$ of $G_0$
 as a complex reductive linear algebraic group.
Write $\ovl{Q}$ for the
 parabolic subgroup of $G$ with Lie algebra $\Bar{\frak{q}}$.

Suppose that $V$ is a $\ovl{Q}$-module and
 use the same letter $V$ for
 the underlying $(\Bar{\frak{q}},L\cap K)$-module.
In Section~\ref{sec:coh}, we define
 the cohomologically induced module
 \[(\Pi_{L\cap K}^{K})_s(U(\frak{g})\otimes_{U(\Bar{\frak{q}})}
 (V\otimes \bb{C}_{2\rho(\frak{u})})), \]
 where $s={\rm dim}(\frak{u}\cap \frak{k})$.

Let $X:=G/\ovl{Q}$ and $Y:=K/(\ovl{Q}\cap K)$,
 which are the partial flag varieties of $G$ and $K$, respectively.
The inclusion map $i : Y\to X$ is a closed immersion.
Let $i_+{\cal O}_Y$ be the push-forward of ${\cal O}_Y$ in
 the category of ${\cal D}$-modules.
We write ${\cal V}_X$ for
 the $G$-equivariant ${\cal O}_X$-module 
 associated with the $\ovl{Q}$-module $V$
 as in Section~\ref{sec:homog}.

The next theorem relates the cohomologically induced module and
 the ${\cal O}_X$-module $i_+{\cal O}_Y\otimes_{{\cal O}_X} {\cal V}_X$.
This theorem is similar to that in \cite{dual},
 but the formulations differ in the following three ways.
First, we assume that $\frak{q}$ is a $\theta$-stable parabolic
 subalgebra and hence $Y$ is a closed subvariety of the partial
 flag variety $X$, 
 while in \cite{dual}, $X$ is a complete flag variety and $Y$ is
 an arbitrary $K$-orbit. 
Second, we assume that $V$ is a $\ovl{Q}$-module and consider
 the ${\cal O}_X$-module
 $i_+{\cal O}_Y\otimes_{{\cal O}_X} {\cal V}_X$ with
 $(\frak{g},K)$-action.
On the other hand, $\frak{l}$ acts as scalars on $V$ 
 and the corresponding twisted ${\cal D}$-module was used in \cite{dual}.
Third, we adopt the functor 
 $P_{\Bar{\frak{q}},L\cap K}^{\frak{g},K}$
 for cohomologically induced modules 
 instead of 
 $I_{\frak{q},L\cap K}^{\frak{g},K}$.
As a result, the dual in the isomorphism in \cite{dual} does not appear 
 in Theorem~\ref{locZuc}.

\begin{thm}
\label{locZuc}
Let $V$ be a $\ovl{Q}$-module. 
Then there is an isomorphism 
\[
(\Pi_{L\cap K}^{K})_{s-i}(U(\frak{g})\otimes_{U(\Bar{\frak{q}})}
(V\otimes \bb{C}_{2\rho(\frak{u})}))\simeq
{\rm H}^{i}(X, i_+{\cal O}_Y\otimes_{{\cal O}_X} {\cal V}_X)
\]
of $(\frak{g},K)$-modules.
\end{thm}

\begin{proof}
Let $\wtl{X}:=G/L$ and $\wtl{Y}:=K/(L\cap K)$.
We have the commutative diagram: 
\begin{align*}
\xymatrix{
\wtl{Y} \ar[r]^*+{\tilde{\imath}} \ar[d] 
& \wtl{X} \ar[d]^*+{\pi}\\
Y   \ar[r]^*+{i}   & X
}
\end{align*}
where the maps are defined canonically.
Denote by ${\cal T}_{\wtl{X}/X}$ the sheaf of local vector fields
 on $\wtl{X}$ tangent to the fiber of $\pi$ and denote by $\Omega_{\wtl{X}/X}$
 the top exterior product of its dual ${\cal T}_{\wtl{X}/X} \spcheck$.
Then $\Omega_{\wtl{X}/X}$ is canonically isomorphic to
 $\Omega_{\wtl{X}}\otimes_{{\cal O}_{\wtl{X}}} \pi^*(\Omega_{X}\spcheck)$.
Consider the complex 
\begin{align*}
{\cal C}^{-d}:=
\tilde{\imath}_+{\cal O}_{\wtl{Y}}\otimes_{{\cal O}_{\wtl{X}}} 
\Omega_{\wtl{X}/X}
\otimes_{{\cal O}_{\wtl{X}}} \bigwedge^{d} {\cal T}_{\wtl{X}/X}.
\end{align*}
We give the boundary map 
${\cal C}^{-d} \to {\cal C}^{-d+1}$
by 
\begin{align*}
&f\otimes \omega \otimes \xi_1 \wedge \dots \wedge \xi_d \\
 \mapsto 
&\sum_i (-1)^{i+1} \bigl(
-\xi_i f \otimes \omega \otimes \xi_1 \wedge \dots 
 \wedge \widehat{\xi_i} \wedge \dots \wedge \xi_d \\
& \hspace*{160pt}
 +f\otimes \omega \xi_i \otimes \xi_1 \wedge \dots 
 \wedge \widehat{\xi_i} \wedge \dots \wedge \xi_d \bigr)\\
+&\sum_{i<j} (-1)^{i+j} \bigl(
f \otimes \omega \otimes [\xi_i,\xi_j] \wedge \xi_1 \wedge \dots \wedge
\widehat{\xi_i}\wedge\dots \wedge \widehat{\xi_j} \wedge
 \dots \wedge \xi_d\bigr), 
\end{align*}
where $f\in \tilde{\imath}_+{\cal O}_{\wtl{Y}}$, $\omega\in\Omega_{\wtl{X}/X}$
 and $\xi_1,\dots,\xi_d\in{\cal T}_{\wtl{X}/X}$. 
Since $\Omega_{\wtl{X}/X}$ and ${\cal T}_{\wtl{X}/X}$ are $G$-equivariant, 
$\frak{g}$ acts on them by differential.
The action of $\frak{g}$ on ${\cal C}^d$
 is given by
 the tensor product of the actions on
 $\tilde{\imath}_+{\cal O}_{\wtl{Y}}$, $\Omega_{\wtl{X}/X}$
 and ${\cal T}_{\wtl{X}/X}$.

By an argument in \cite{dual}, we have
 a quasi-isomorphism of the complexes 
 $\pi_*{\cal C}^\bullet \simeq (i_+{\cal O}_Y)[s]$,
 which respects $\frak{g}$-actions.
Here, $[s]$ denotes the shift by $s$.
Then the projection formula gives a quasi-isomorphism of complexes of
 ${\cal O}_X$-modules 
\begin{align*}
 \pi_*({\cal C}^\bullet 
 \otimes_{\pi^{-1}{\cal O}_{X}} \pi^{-1}{\cal V}_{X})
 \simeq i_+{\cal O}_Y \otimes_{{\cal O}_{X}} {\cal V}_{X}[s].
\end{align*}

The isomorphism
 $\Omega_{\wtl{X}/X}\simeq\Omega_{\wtl{X}}\otimes_{{\cal O}_{\wtl{X}}}
 \pi^*(\Omega_X\spcheck)$ 
gives 
\[{\cal C}^{-d}
 \otimes_{\pi^{-1}{\cal O}_{X}} \pi^{-1}{\cal V}_{X}
\simeq 
 \tilde{\imath}_+{\cal O}_{\wtl{Y}}
 \otimes_{{\cal O}_{\wtl{X}}} \Omega_{\wtl{X}}
 \otimes_{{\cal O}_{\wtl{X}}} \bigwedge^{d} {\cal T}_{\wtl{X}/X}
 \otimes_{{\cal O}_{\wtl{X}}}
 \pi^*({\cal V}_X \otimes_{{\cal O}_X} \Omega_X\spcheck).
\]
The boundary map $\partial$ on the right side is given by 
\begin{align*}
&f\otimes \xi_1\wedge\cdots\wedge\xi_d\otimes v \\
\mapsto\ &\sum_i(-1)^{i+1}\bigl( f\xi_i\otimes 
 \xi_1\wedge\cdots\wedge\widehat{\xi_i}\wedge\cdots\wedge\xi_d \otimes v \\
& \hspace*{160pt}
 -f \otimes \xi_1\wedge\cdots\wedge\widehat{\xi_i}
 \wedge\cdots\wedge\xi_d \otimes \xi_i v \bigr) \\
+\ &\sum_{i<j} (-1)^{i+j} \bigl(f\otimes [\xi_i,\xi_j]\wedge \xi_1\wedge 
 \cdots \wedge \widehat{\xi_i}\wedge\cdots \wedge\widehat{\xi_j} \wedge
 \cdots \wedge \xi_d \otimes v \bigr)
\end{align*}
 for $f\in \tilde{\imath}_+{\cal O}_{\wtl{Y}}\otimes\Omega_{\wtl{X}}$,
 $\xi_1,\dots,\xi_d \in {\cal T}_{\wtl{X}/X}$, and 
 $v\in \pi^*({\cal V}_X \otimes_{{\cal O}_X} \Omega_X\spcheck)$.
Here, the action of $\xi_i\in{\cal T}_{\wtl{X}/X}$ on
 $\tilde{\imath}_+{\cal O}_{\wtl{Y}}\otimes\Omega_{\wtl{X}}$
 is defined by the right ${\cal D}_{\wtl{X}}$-module
 structure of $\tilde{\imath}_+{\cal O}_{\wtl{Y}}\otimes\Omega_{\wtl{X}}$, 
 and the action of $\xi_i$ on
\[\pi^*({\cal V}_X \otimes_{{\cal O}_X} \Omega_X\spcheck)
 :={\cal O}_{\wtl{X}}\otimes_{\pi^{-1}{\cal O}_X}
 \pi^{-1}({\cal V}_X \otimes_{{\cal O}_X} \Omega_X\spcheck)\]
 is given by the action on the first factor 
${\cal O}_{\wtl{X}}$ of the right side.
Since $\wtl{X}$ is affine, we have an isomorphism of $(\frak{g},K)$-modules
\begin{align*}
&{\rm H}^{i-s}\Bigl(\Gamma\bigl(\wtl{X},\,  \tilde{\imath}_+{\cal O}_{\wtl{Y}}
 \otimes_{{\cal O}_{\wtl{X}}} \Omega_{\wtl{X}}
 \otimes_{{\cal O}_{\wtl{X}}} \bigwedge^{\bullet} {\cal T}_{\wtl{X}/X}
 \otimes_{{\cal O}_{\wtl{X}}}
 \pi^*({\cal V}_X \otimes_{{\cal O}_X} \Omega_X\spcheck)\bigr)\Bigr)\\
\simeq \ 
&{\rm H}^i(X,\, i_+{\cal O}_Y \otimes_{{\cal O}_{X}} {\cal V}_{X}).
\end{align*}

We now compute the cohomologically induced module
 $(\Pi_{L\cap K}^{K})_{s-i}(U(\frak{g})\otimes_{U(\Bar{\frak{q}})}
 (V\otimes \bb{C}_{2\rho(\frak{u})}))$. 
The standard complex of $\Bar{\frak{u}}$ is the complex
 $U(\Bar{\frak{u}})\otimes\bigwedge^{\bullet}\Bar{\frak{u}}$
 with the boundary map
\begin{align*}
D \otimes 
 \xi_1\wedge \dots \wedge \xi_d
\mapsto\ &
\sum_{i=1}^d (-1)^{i+1} \bigl( D\xi_i \otimes 
 \xi_1\wedge \dots \wedge \widehat{\xi_i} \wedge
 \dots \wedge \xi_d \bigr)\\
+&\sum_{i<j}(-1)^{i+j} \bigl( D\otimes 
[\xi_i,\xi_j]\wedge \xi_1 \wedge \dots \wedge \widehat{\xi_i}
\wedge \dots \wedge \widehat{\xi_j} \wedge \dots \wedge \xi_d \bigr) 
\end{align*}
for $D\in U(\Bar{\frak{u}})$ and $\xi_1,\dots,\xi_d\in \Bar{\frak{u}}$.
This gives a left resolution of the trivial $\Bar{\frak{u}}$-module:
\[
U(\Bar{\frak{u}})\otimes\bigwedge^{\bullet}\Bar{\frak{u}}\to \bb{C}.
\]
Since $U(\bar{\frak{u}})\simeq U(\bar{\frak{q}})/U(\bar{\frak{q}})\frak{l}$, 
 we have an isomorphism 
\[
U(\Bar{\frak{q}})\otimes_{U(\frak{l})}\bigwedge^{d}\Bar{\frak{u}}\simeq
U(\Bar{\frak{u}})\otimes_\bb{C}\bigwedge^{d}\Bar{\frak{u}}.\]
Hence we have a left resolution of the trivial 
$(\Bar{\frak{q}},L\cap K)$-modules:
\[
U(\Bar{\frak{q}})\otimes_{U(\frak{l})}\bigwedge^{d}\Bar{\frak{u}}\to \bb{C}.
\]
By taking tensor product with
 $V\otimes \bb{C}_{2\rho(\frak{u})}$,
 we get a resolution of the $(\Bar{\frak{q}},L\cap K)$-module
 $V\otimes\bb{C}_{2\rho(\frak{u})}$:
\[
U(\Bar{\frak{q}})\otimes_{U(\frak{l})}(\bigwedge^{\bullet}\Bar{\frak{u}}
\otimes V\otimes \bb{C}_{2\rho(\frak{u})})
\simeq 
(U(\Bar{\frak{q}})\otimes_{U(\frak{l})}\bigwedge^{\bullet}\Bar{\frak{u}})
\otimes (V\otimes \bb{C}_{2\rho(\frak{u})})
\to V\otimes \bb{C}_{2\rho(\frak{u})}.
\]
Therefore, we have a resolution of the
 $(\frak{g},L\cap K)$-module
 $U(\frak{g})\otimes_{U(\Bar{\frak{q}})}
 (V\otimes\bb{C}_{2\rho(\frak{u})})$:
\[
U(\frak{g})\otimes_{U(\frak{l})} 
(\bigwedge^\bullet
  \Bar{\frak{u}}\otimes V\otimes\bb{C}_{2\rho(\frak{u})})
\to U(\frak{g})\otimes_{U(\Bar{\frak{q}})}
 (V\otimes\bb{C}_{2\rho(\frak{u})}),
\]
where the boundary map $\partial'$ is given by
\begin{align*}
&D \otimes 
 \xi_1\wedge \dots \wedge \xi_d \otimes v\\
\mapsto 
\ &\sum_{i=1}^d (-1)^{i+1} \bigl(D\xi_i \otimes 
 \xi_1\wedge \dots \wedge \widehat{\xi_i} \wedge \dots 
\wedge \xi_d \otimes v \\
&\hspace*{160pt}- D \otimes 
 \xi_1\wedge \dots \wedge \widehat{\xi_i} \wedge \dots 
\wedge \xi_d \otimes \xi_i v 
\bigr)\\
+\ &\sum_{i<j}(-1)^{i+j}\bigl(D\otimes 
[\xi_i,\xi_j]\wedge \xi_1 \wedge \dots \wedge \widehat{\xi_i}
\wedge \dots \wedge \widehat{\xi_j} \wedge \dots \wedge \xi_d
\otimes v\bigr) 
\end{align*}
for $D\in U(\frak{g})$, $\xi_1,\dots,\xi_d\in \Bar{\frak{u}}$,
 and $v\in V\otimes \bb{C}_{2\rho(\frak{u})}$.

\begin{lem}
For any $(\frak{l},L\cap K)$-module $W$, 
 the $(\frak{g}, L\cap K)$-module $U(\frak{g})\otimes_{U(\frak{l})} W$
 is $\Pi_{L\cap K}^K$-acyclic.
\end{lem}
\begin{proof}
By \cite[Proposition 2.115]{KnVo}, 
 $(P_{\frak{g},L\cap K}^{\frak{g},K})_j
 (U(\frak{g})\otimes_{U(\frak{l})} W)
 \simeq(P_{\frak{k},L\cap K}^{\frak{k},K})_j
 (U(\frak{g})\otimes_{U(\frak{l})} W)$
 as a $K$-module.
Hence
 it is enough to show that
 $(P_{\frak{k},L\cap K}^{\frak{k},K})_j
 (U(\frak{g})\otimes_{U(\frak{l})} W)=0$
 for $j>0$.
Let $U_p(\frak{g})$ be the standard filtration of $U(\frak{g})$ and
 let $U_p'(\frak{g}):=U(\frak{k})U_p(\frak{g})U(\frak{l})\subset U(\frak{g})$
 for $p\geq 0$.
Then $U_p'(\frak{g})\otimes_{U(\frak{l})} W$ is
 a filtration of
 the $(\frak{k},L\cap K)$-module
 $U(\frak{g})\otimes_{U(\frak{l})} W$ and
it follows that
\[
U_p'(\frak{g})\otimes_{U(\frak{l})} W \,/\,
U_{p-1}'(\frak{g})\otimes_{U(\frak{l})} W\simeq
U(\frak{k})\otimes_{U(\frak{l}\cap \frak{k})} 
(S^p(\frak{g}/(\frak{l}+\frak{k}))\otimes W).
\]
Since 
\[{\Hom}_{\frak{k},L\cap K}(
U(\frak{k})\otimes_{U(\frak{l}\cap \frak{k})} 
(S^p(\frak{g}/(\frak{l}+\frak{k}))\otimes W), \,\cdot\,)
\simeq {\Hom}_{L\cap K}(S^p(\frak{g}/(\frak{l}+\frak{k}))\otimes W, \,\cdot\,),
\]
we see that 
$U_p'(\frak{g})\otimes_{U(\frak{l})} W \,/\,
 U_{p-1}'(\frak{g})\otimes_{U(\frak{l})} W$
 is a projective $(\frak{k},L\cap K)$-module.
Then we see inductively that 
$U_p'(\frak{g})\otimes_{U(\frak{l})} W$ is
 also a projective $(\frak{k},L\cap K)$-module
 and in particular
 $P_{\frak{k},L\cap K}^{\frak{k},K}$-acyclic.
As a consequence, 
\begin{align*}
(P_{\frak{k},L\cap K}^{\frak{k},K})_j(U(\frak{g})\otimes_{U(\frak{l})} W)
&=(P_{\frak{k},L\cap K}^{\frak{k},K})_j\varinjlim_p 
 (U_p'(\frak{g})\otimes_{U(\frak{l})} W)\\
&=\varinjlim_p  (P_{\frak{k},L\cap K}^{\frak{k},K})_j(
 U_p'(\frak{g})\otimes_{U(\frak{l})} W)
=0
\end{align*}
for $j>0$.
\end{proof}

From the lemma, we conclude that 
\[(\Pi_{L\cap K}^{K})_{s-i}
(U(\frak{g})\otimes_{U(\Bar{\frak{q}})}(V\otimes \bb{C}_{2\rho(\frak{u})}))
\simeq {\rm H}^{i-s}(\Pi_{L\cap K}^{K}(
U(\frak{g})\otimes_{U(\frak{l})}
(\bigwedge^\bullet \bar{\frak{u}}
\otimes V\otimes \bb{C}_{2\rho(\frak{u})}))).\]
To complete the proof of Theorem \ref{locZuc},
 it is enough to give an isomorphism of
 the complexes of $(\frak{g},K)$-modules: 
\begin{align}
\label{eqn:cpxiso}
&\Gamma\bigl(\wtl{X},\ 
 \tilde{\imath}_+{\cal O}_{\wtl{Y}}\otimes_{{\cal O}_{\wtl{X}}}
 \Omega_{\wtl{X}} \otimes_{{\cal O}_{\wtl{X}}} 
 \bigwedge^\bullet {\cal T}_{\wtl{X}/X}\otimes_{{\cal O}_{\wtl{X}}}
 \pi^*({\cal V}_X\otimes_{{\cal O}_X} \Omega_X\spcheck)\bigr)\\ \nonumber
\xrightarrow{\sim}\ 
 &R(\frak{g},K)\otimes_{R(\frak{l},L\cap K)}
 (\bigwedge^\bullet \bar{\frak{u}} \otimes 
 V\otimes \bb{C}_{2\rho(\frak{u})}).
\end{align}

Let $o:=e(L\cap K)\in \wtl{Y}$ be the base point and
 $i_o:\{o\}\to \wtl{Y}$ the immersion.
Define the complex of left ${\cal D}_{\wtl{Y}}$-modules
\[\tilde{\imath} {\,}^*{\cal D}_{\wtl{X}} 
 \otimes_{\tilde{\imath}{\,}^{-1}{\cal O}_{\wtl{X}}}
 \tilde{\imath}{\,}^{-1}\bigl(
 \bigwedge^\bullet {\cal T}_{\wtl{X}/X}\otimes_{{\cal O}_{\wtl{X}}}
 \pi^*({\cal V}_X\otimes_{{\cal O}_X} \Omega_X\spcheck)\bigr),\]
where the boundary map 
\begin{align*}
\partial:\ 
&\tilde{\imath} {\,}^*{\cal D}_{\wtl{X}} 
 \otimes_{\tilde{\imath}{\,}^{-1}{\cal O}_{\wtl{X}}}
 \tilde{\imath}{\,}^{-1}\bigl(
 \bigwedge^d {\cal T}_{\wtl{X}/X}\otimes_{{\cal O}_{\wtl{X}}}
 \pi^*({\cal V}_X\otimes_{{\cal O}_X} \Omega_X\spcheck)\bigr)\\
\to\ &\tilde{\imath} {\,}^*{\cal D}_{\wtl{X}} 
 \otimes_{\tilde{\imath}{\,}^{-1}{\cal O}_{\wtl{X}}}
 \tilde{\imath}{\,}^{-1}\bigl(
 \bigwedge^{d-1} {\cal T}_{\wtl{X}/X}\otimes_{{\cal O}_{\wtl{X}}}
 \pi^*({\cal V}_X\otimes_{{\cal O}_X} \Omega_X\spcheck)\bigr)
\end{align*}
is given by
\begin{align}
\label{bound}
&\partial(D\otimes \xi_1\wedge\cdots\wedge\xi_d\otimes v) \\ \nonumber
:=\ &\sum_i(-1)^{i+1} \bigl( D\xi_i\otimes 
 \xi_1\wedge\cdots\wedge\widehat{\xi_i}\wedge\cdots\wedge
 \xi_d \otimes v\\ \nonumber
&\hspace*{160pt} -D \otimes \xi_1\wedge\cdots\wedge\widehat{\xi_i}
 \wedge\cdots\wedge\xi_d \otimes \xi_i v \bigr) \\ \nonumber
+\ &\sum_{i<j} (-1)^{i+j} \bigl( D\otimes [\xi_i,\xi_j]\wedge \xi_1\wedge 
 \cdots \wedge \widehat{\xi_i}\wedge\cdots \wedge\widehat{\xi_j} \wedge
 \cdots \wedge \xi_d \otimes v \bigr).
\end{align}
for $D\in \tilde{\imath}{\,}^*{\cal D}_{\wtl{X}}$, 
 $\xi,\cdots,\xi_d\in {\cal T}_{\wtl{X}/X}$, and 
 $v\in \pi^*({\cal V}_X\otimes_{{\cal O}_X} \Omega_X\spcheck)$.
In view of the proof of Lemma~\ref{affloc},
 we have only to see that the pull-back $(i_o)^*$ sends
 the complex
\begin{align*}
\tilde{\imath} {\,}^*{\cal D}_{\wtl{X}} 
 \otimes_{\tilde{\imath}{\,}^{-1}{\cal O}_{\wtl{X}}}
 \tilde{\imath}{\,}^{-1}\bigl(
 \bigwedge^\bullet {\cal T}_{\wtl{X}/X}\otimes_{{\cal O}_{\wtl{X}}}
 \pi^*({\cal V}_X\otimes_{{\cal O}_X} \Omega_X\spcheck)\bigr)
\end{align*}
to $U(\frak{g})\otimes_{U(\Bar{\frak{q}})} 
(\bigwedge^{\bullet}
 \bar{\frak{u}}\otimes V\otimes \bb{C}_{2\rho(\frak{u})})$.

Write $V^d:=\bigwedge^{d}\Bar{\frak{u}}\otimes V\otimes
 \bb{C}_{2\rho(\frak{u})}$ for simplicity.
Since $\bigwedge^{d} {\cal T}_{\wtl{X}/X}\otimes_{{\cal O}_{\wtl{X}}}
 \pi^*({\cal V}_X\otimes_{{\cal O}_X} \Omega_X\spcheck)$
 is isomorphic to
 the ${\cal O}_{\wtl{X}}$-module ${\cal V}_{\wtl{X}}^d$ associated with
 the $L$-module $V^d$,
 it follows that
\begin{align*}
&(i_o)^*\Bigl(\tilde{\imath} {\,}^*{\cal D}_{\wtl{X}} 
 \otimes_{\tilde{\imath}{\,}^{-1}{\cal O}_{\wtl{X}}}
 \tilde{\imath}{\,}^{-1}\bigl(
 \bigwedge^{d} {\cal T}_{\wtl{X}/X}\otimes_{{\cal O}_{\wtl{X}}}
 \pi^*({\cal V}_X\otimes_{{\cal O}_X} \Omega_X\spcheck)\bigr)\Bigr)\\
\simeq \ 
&(i_o)^*\bigl(\tilde{\imath} {\,}^*{\cal D}_{\wtl{X}} 
 \otimes_{\tilde{\imath}{\,}^{-1}{\cal O}_{\wtl{X}}}
 \tilde{\imath}{\,}^{-1}
 {\cal V}_{\wtl{X}}^d\bigr)\\
\simeq \ 
&U(\frak{g})\otimes_{U(\frak{l})} V^d
\end{align*}
as in the proof of Lemma \ref{affloc}.
Therefore, 
$\tilde{\imath} {\,}^*{\cal D}_{\wtl{X}} 
 \otimes_{\tilde{\imath}{\,}^{-1}{\cal O}_{\wtl{X}}}
 \tilde{\imath}{\,}^{-1}
 {\cal V}_{\wtl{X}}^d$
 is isomorphic to the $K$-equivariant ${\cal O}_{\wtl{Y}}$-module
 associated with the $(L\cap K)$-module $U(\frak{g})\otimes_{U(\frak{l})} V^d$.
Via this isomorphism, we view a section
\[
f\in  \tilde{\imath} {\,}^*{\cal D}_{\wtl{X}} 
 \otimes_{\tilde{\imath}{\,}^{-1}{\cal O}_{\wtl{X}}}
 \tilde{\imath}{\,}^{-1}
 {\cal V}_{\wtl{X}}^d
\]
as a regular function on an open set of $K$ that takes values in
 $U(\frak{g}) \otimes_{U(\frak{l})} V^d$. 
Write $f(e)\in U(\frak{g}) \otimes_{U(\frak{l})} V^d$
 for the evaluation at the identity $e\in K$.
The boundary map \eqref{bound} is 
 ${\cal O}_{\wtl{Y}}$-linear and hence induces 
 an operator 
\[\partial_e: U(\frak{g})\otimes_{U(\frak{l})} V^d
 \to U(\frak{g})\otimes_{U(\frak{l})} V^{d-1}\]
such that
$\partial_e(f(e))=(\partial f)(e)$ for every 
$f\in  \tilde{\imath} {\,}^*{\cal D}_{\wtl{X}} 
 \otimes_{\tilde{\imath}{\,}^{-1}{\cal O}_{\wtl{X}}}
 \tilde{\imath}{\,}^{-1}
 {\cal V}_{\wtl{X}}^d.$
It is enough to show that $\partial_e=\partial'$ on
 $U(\frak{g})\otimes_{U(\frak{l})} V^\bullet$.

Put $Z:=(\ovl{U}\cdot L)/L\subset G/L=\wtl{X}$
 and write $i_Z : Z\to \wtl{X}$ for the inclusion map
 so that $i_Z(Z)=\pi^{-1}(\{o\})$.
Then under the isomorphism $Z\simeq \ovl{U}$,
 there is a canonical
 isomorphism of $\ovl{U}$-equivariant ${\cal O}$-modules
 $\iota:i_Z^*{\cal T}_{\wtl{X}/X}\simeq {\cal T}_{\ovl{U}}$.

For $\xi_1,\dots,\xi_d\in\Bar{\frak{u}}$
 and $v\in V\otimes\bb{C}_{2\rho(\frak{u})}$, 
 put
\[m:=\xi_1\wedge\cdots\wedge\xi_d\otimes v\,\in\, V^d.\]
We will choose sections $\wtl{\xi_i} \in {\cal T}_{\wtl{X}/X}$
 and $\wtl{v}\in \pi^*({\cal V}_X\otimes_{{\cal O}_X} \Omega_X\spcheck)$
 on a neighborhood of the base point $o\in \wtl{X}$ in the following way.
Take $\wtl{\xi_i}\in {\cal T}_{\wtl{X}/X}$ 
 such that $\wtl{\xi_i}|_{Z}
 \in i_Z^*{\cal T}_{\wtl{X}/X}$
 corresponds to 
 $(\xi_i)_{\ovl{U}}^R$ under $\iota$.
The $G$-equivariant ${\cal O}_{X}$-module
 ${\cal V}_X \otimes_{{\cal O}_X} \Omega_X\spcheck$ is isomorphic to
 the ${\cal O}_{X}$-module associated with
 the $\ovl{Q}$-module $V\otimes\bb{C}_{2\rho(\frak{u})}$.
Hence $f\in  \pi^*({\cal V}_X \otimes_{{\cal O}_X} \Omega_X\spcheck)$
 is identified with a $(V\otimes \bb{C}_{2\rho(\frak{u})})$-valued
 regular function on an open set of $\wtl{X}$ satisfying 
 $f(gq)=q^{-1}\cdot f(g)$ for $g\in G$ and $q\in \ovl{Q}$.
With this identification, 
 we take a section $\wtl{v}'\in
 {\cal V}_X \otimes_{{\cal O}_X} \Omega_X\spcheck$
 on a neighborhood of $o$ 
 such that $\wtl{v}'(e)=v$.
Define the section
 $\wtl{v}\in \pi^*({\cal V}_X \otimes_{{\cal O}_X} \Omega_X\spcheck)$ 
as 
\[
\wtl{v}:=1\otimes \wtl{v}'\in
 {\cal O}_{\wtl{X}}\otimes_{\pi^{-1}{\cal O}_X}
\pi^{-1}({\cal V}_X \otimes_{{\cal O}_X} \Omega_X\spcheck).
\]
and define the section $\wtl{m}\in {\cal V}^d_{\wtl{X}}$
 in a neighborhood of $o$ as
\begin{align*}
\wtl{m}:=\wtl{\xi_1}\wedge\cdots\wedge\wtl{\xi_d}
 \otimes \wtl{v}
 \,\in\, {\cal V}^d_{\wtl{X}}.
\end{align*}
Then 
\[1\otimes \wtl{m}\in
\tilde{\imath} {\,}^*{\cal D}_{\wtl{X}} 
 \otimes_{\tilde{\imath}{\,}^{-1}{\cal O}_{\wtl{X}}}
 \tilde{\imath}{\,}^{-1}
 {\cal V}_{\wtl{X}}^d
\]
satisfies
$(1\otimes \wtl{m}) (e) =1\otimes m$.

We have
\begin{align*}
&\partial(1\otimes \wtl{m})\\
=\ 
 &\sum_i(-1)^{i+1} \Bigl((\xi_i)_{\wtl{X}} \otimes 
 \wtl{\xi_1}\wedge\cdots\wedge\widehat{\wtl{\xi_i}}\wedge\cdots\wedge
 \wtl{\xi_d}\otimes \wtl{v} \\
&\hspace*{160pt} -1 \otimes 
 \wtl{\xi_1}\wedge\cdots\wedge\widehat{\wtl{\xi_i}}\wedge\cdots\wedge
 \wtl{\xi_d}\otimes \wtl{\xi_i}\wtl{v}
 \Bigr)\\
 +\ &\sum_{i<j}(-1)^{i+j}\, 
 \bigl(1\otimes [\wtl{\xi_i},\wtl{\xi_j}]\wedge
 \wtl{\xi_1}\wedge\cdots
 \wedge\widehat{\wtl{\xi_i}}\wedge\cdots\wedge\widehat{\wtl{\xi_j}}\wedge
 \cdots\wedge\wtl{\xi_d}\otimes \wtl{v}\bigr)
\end{align*}
and
\begin{align*}
&\partial'(1\otimes m)\\
=\ & \sum_i (-1)^{i+1} \bigl(\xi_i \otimes
 \xi_1\wedge\cdots\wedge\widehat{\xi_i}\wedge\cdots\wedge\xi_d
 \otimes v
- 1 \otimes
 \xi_1\wedge\cdots\wedge\widehat{\xi_i}\wedge\cdots\wedge\xi_d
 \otimes \xi_i v\bigr) \\
+\ &\sum_{i<j}(-1)^{i+j}\, 
 \bigl( 1 \otimes [\xi_i,\xi_j]\wedge
 \xi_1\wedge\cdots
 \wedge\widehat{\xi_i}\wedge\cdots\wedge\widehat{\xi_j}\wedge
 \cdots\wedge\xi_d \otimes v\bigr).
\end{align*}

Since $\wtl{\xi_i}|_{Z}$ corresponds to $(\xi_i)^R_{\ovl{U}}$, 
 the tangent vectors at the base point $o$ of the vector fields
 $\wtl{\xi_i}$ and $(\xi_i)_{\wtl{X}}$ have the relation:
 $(\wtl{\xi_i})_{o}
 =-((\xi_i)_{\wtl{X}})_{o}$.
Recall that the $\frak{g}$-actions on
 ${\cal T}_{\wtl{X}/X}$ and
 $\pi^*({\cal V}_{\wtl X}\otimes\Omega_{\wtl{X}})$
 are defined as the differentials of the $G$-equivariant structures on them.
Our choice implies that $\wtl{\xi_j}|_{Z}$ is left $\ovl{U}$-invariant
 and hence $\xi_i\cdot\wtl{\xi_j}|_{Z}=0$.
We therefore have
\[(1\otimes\xi_i
 (\wtl{\xi_1}\wedge\cdots\wedge\widehat{\wtl{\xi_i}}\wedge
 \cdots\wedge\wtl{\xi_d})\otimes \wtl{v})(e)
 =0.
\]
In addition, our choice of $\wtl{v}$ implies that
 ${\cal T}_{\wtl{X}/X} \wtl{v}=0$ and
 $(\xi_i \wtl{v})(e)=\xi_i v$.
As a result,  
\begin{align*}
&\Bigl((\xi_i)_{\wtl{X}} \otimes 
 \wtl{\xi_1}\wedge\cdots\wedge\widehat{\wtl{\xi_i}}\wedge\cdots\wedge
 \wtl{\xi_d}\otimes \wtl{v}
 -1 \otimes 
 \wtl{\xi_1}\wedge\cdots\wedge\widehat{\wtl{\xi_i}}\wedge\cdots\wedge
 \wtl{\xi_d}\otimes \wtl{\xi_i}\wtl{v}
 \Bigr)(e)\\ \nonumber
=\ &\Bigl((\xi_i)_{\wtl{X}} \otimes 
 \wtl{\xi_1}\wedge\cdots\wedge\widehat{\wtl{\xi_i}}\wedge\cdots\wedge
 \wtl{\xi_d}\otimes \wtl{v}
 \Bigr)(e)\\ \nonumber
=\ &\bigl(\xi_i(1\otimes \wtl{\xi_1}\wedge\cdots\wedge\widehat{\wtl{\xi_i}}
 \wedge\cdots\wedge\wtl{\xi_d}\otimes \wtl{v})\bigr)(e)
- (1\otimes \wtl{\xi_1}\wedge\cdots\wedge\widehat{\wtl{\xi_i}}
 \wedge\cdots\wedge\wtl{\xi_d}\otimes \xi_i\wtl{v})(e) \\ \nonumber
=\ &\xi_i\bigl((1\otimes \wtl{\xi_1}\wedge\cdots\wedge\widehat{\wtl{\xi_i}}
 \wedge\cdots\wedge\wtl{\xi_d}\otimes \wtl{v})(e)\bigr)
- (1\otimes \wtl{\xi_1}\wedge\cdots\wedge\widehat{\wtl{\xi_i}}
 \wedge\cdots\wedge\wtl{\xi_d}\otimes \xi_i\wtl{v})(e) \\ \nonumber
=\ & \xi_i\otimes
 \xi_1\wedge\cdots\wedge\widehat{\xi_i}\wedge\cdots\wedge\xi_d
 \otimes v
 - 1\otimes \xi_1\wedge\cdots\wedge\widehat{\xi_i}\wedge\cdots\wedge\xi_d
 \otimes \xi_i v.
\end{align*}
Moreover, $[\wtl{\xi_i},\wtl{\xi_j}]|_Z$ corresponds to
 $[(\xi_i)_{\ovl{U}}^R,(\xi_j)_{\ovl{U}}^R]
 =([\xi_i,\xi_j])_{\ovl{U}}^R$.
Hence 
\begin{align*}
 &\bigl(1\otimes [\wtl{\xi_i},\wtl{\xi_j}]\wedge
 \wtl{\xi_1}\wedge\cdots
 \wedge\widehat{\wtl{\xi_i}}\wedge\cdots\wedge\widehat{\wtl{\xi_j}}\wedge
 \cdots\wedge\wtl{\xi_d}\otimes \wtl{v}\bigr)(e)\\
=\ &1\otimes [\xi_i, \xi_j]\wedge
 \xi_1\wedge\cdots
 \wedge\widehat{\xi_i}\wedge\cdots\wedge\widehat{\xi_j}\wedge
 \cdots\wedge \xi_d\otimes v.
\end{align*}

We thus conclude that
\[\partial_e(1\otimes m)=\partial_e((1\otimes \wtl{m})(e))
 =(\partial(1\otimes \wtl{m}))(e)=\partial' (1\otimes m).\]
Since $\partial_e$ and $\partial'$ commute with $\frak{g}$-actions,
 $\partial_e=\partial'$.
Therefore, we obtain an isomorphism \eqref{eqn:cpxiso}
 and prove the theorem.
\end{proof}


\section{Construction of Parabolic Subalgebras}\label{sec:constparab}
Let $G_0$ be a connected real linear reductive Lie group
 with Lie algebra $\frak{g}_0$ and $\sigma$ an involution of $G_0$. 
Let $G'_0$ be the identity component of the fixed point set $G_0^\sigma$.
There exists a Cartan involution $\theta$ of $G_0$ that
 commutes with $\sigma$.
The corresponding maximal compact subgroups of $G_0$ and $G'_0$
 are written as $K_0:=G_0^{\theta}$ and $K_0':=(G_0')^{\theta}$, respectively.
The Cartan decompositions are written as
 $\frak{g}_0=\frak{k}_0+\frak{p}_0$ and $\frak{g}'_0=\frak{k}'_0+\frak{p}'_0$.
We denote by $\frak{g},\frak{g'},\frak{k}$, etc.\ the complexifications
 of $\frak{g}_0,\frak{g}'_0,\frak{k}_0$, etc.
Let $\sigma$ and $\theta$ also denote the induced actions
 on $\frak{g}_0$
 and their complex linear extensions to $\frak{g}$.

\begin{de}
{\rm 
Let $V$ be a $(\frak{g}',K')$-module.
We say that $V$ is {\it discretely decomposable}
 if $V$ admits a filtration $\{V_p\}_{p\in\bb{N}}$
 such that $V=\bigcup_{p\in\bb{N}} V_p$ and
 $V_p$ is of finite length as a $(\frak{g}',K')$-module
 for each $p\in\bb{N}$.
}
\end{de}

If $V$ is unitarizable and discretely decomposable,
 then $V$ is an algebraic direct sum of irreducible
 $(\frak{g}',K')$-modules (see \cite[Lemma 1.3]{kob98ii}).

\begin{de}
\label{openparab}
{\rm 
Suppose that $\frak{q}$ is a $\theta$-stable parabolic
 subalgebra of $\frak{g}$.
We say that $\frak{q}$ is {\it $\sigma$-open}
 if $\frak{q}\cap\frak{k}+\frak{k}'=\frak{k}.$
}
\end{de}

\begin{rem}
{\rm 
If $\frak{q}$ is a $\theta$-stable parabolic subalgebra of $\frak{g}$,
 there exists a $\sigma$-open $\theta$-stable
 parabolic subalgebra that is conjugate to $\frak{q}$
 under the adjoint action of $K_0$.
}
\end{rem}

We write
 ${\cal{N}}_{\frak{g}}$ and ${\cal{N}}_{\frak{g'}}$
 for the nilpotent cones of $\frak{g}$ and $\frak{g}'$, respectively.
Let $\pr_{\frak{g}\to\frak{g'}}$ denote the projection
 from $\frak{g}$ onto $\frak{g}'$ along $\frak{g}^{-\sigma}$.

\begin{thm}
\label{parab}
Let $(G_0,G'_0)$ be a symmetric pair of connected
 real linear reductive Lie groups defined by an involution $\sigma$.
Let $\frak{q}$ be a $\sigma$-open $\theta$-stable
 parabolic subalgebra of $\frak{g}$.
Then the following three conditions are equivalent.

\begin{enumerate}
\item[{\rm (i)}]
$A_\frak{q}(\lambda)$ is
 nonzero and discretely decomposable
 as a $(\frak{g}',K')$-module for some $\lambda$
 in the weakly fair range.
\item[$({\rm ii})$]
$A_\frak{q}(\lambda)$ is
 discretely decomposable
 as a $(\frak{g}',K')$-module for any $\lambda$
 in the weakly fair range.
\item[{\rm (iii)}]
Put
 $\frak{q}':=N_{\frak{k}'}(\frak{q}\cap \frak{p}')
+(\frak{q}\cap\frak{p}')$,
 where $N_{\frak{k}'}(\frak{q}\cap \frak{p}')$ is the normalizer of
 $\frak{q}\cap \frak{p}'$ in $\frak{k}'$.
Then $\frak{q}'$ is a $\theta$-stable parabolic subalgebra of $\frak{g}'$.
\end{enumerate}
\end{thm}

The proof is based on the following criterion for the discrete decomposability
({\cite[Theorem~4.2]{kob98ii}}).

\begin{fact}
\label{discdecomp}
In the setting of Theorem~\ref{parab},
 the following conditions are equivalent.
\begin{itemize}
\item[$({\rm i})$]
$A_\frak{q}(\lambda)$ is
 nonzero and discretely decomposable
 as a $(\frak{g}',K')$-module for some $\lambda$
 in the weakly fair range.
\item[$({\rm ii})$]
$A_\frak{q}(\lambda)$ is
 discretely decomposable
 as a $(\frak{g}',K')$-module for any $\lambda$
 in the weakly fair range.
\item[$({\rm iv})$]
$\pr_{\frak{g}\to\frak{g}'}(\frak{u}\cap\frak{p})\subset{\cal{N}}_{\frak{g}'}$
 for the nilradical $\frak{u}$ of $\frak{q}$.
\end{itemize}
\end{fact}

We use the following lemma for the proof of Theorem~\ref{parab}.

\begin{lem}
\label{proj}
Let $V$ be a finite-dimensional vector space with
 a non-degenerate symmetric bilinear form.
For subspaces $V_1\subset V_2\subset V$, we denote by $V_1^{\perp V_2}$
 the set of all vectors in $V_2$ that are orthogonal to $V_1$.

Suppose that $X$ is a subspace of $V$ such that
 $V=X\oplus {X}^{\perp V}$.
Let $p$ be the projection onto $X$ along $X^{\perp V}$.
Then for any subspace $W\subset V$, it follows that
\[(W\cap X)^{\perp X}=p(W^{\perp V}).\]
\end{lem}

\begin{proof}
We have
\begin{align*}
(W\cap X)^{\perp X}=(W\cap X)^{\perp {V}}\cap X 
=(W^{\perp {V}}+{X}^{\perp {V}})\cap X 
=p(W^{\perp V}),
\end{align*}
so the assertion is verified.
\end{proof}

\begin{proof}[Proof of Theorem~\ref{parab}]
First of all, $\frak{q}'$ defined in (iii) is a subalgebra of $\frak{g}$ 
because $[\frak{q}\cap \frak{p}',\frak{q}\cap \frak{p}']\subset
\frak{q}\cap \frak{k}'\subset N_{\frak{k}'}(\frak{q}\cap \frak{p}')$.

Choose an invariant symmetric bilinear form $\langle\cdot,\cdot\rangle$
on $\frak{g}$ such that the subspaces 
$\frak{k}', \frak{k}^{-\sigma}, \frak{p}'$, and
$\frak{p}^{-\sigma}$
are mutually orthogonal.
We use the letter ${}^\perp$ for orthogonal spaces with respect to 
$\langle\cdot,\cdot\rangle$ as in Lemma~\ref{proj}.

It is enough to prove the equivalence of  
(iii) and (iv)
by Fact~\ref{discdecomp}.

Assume that (iii) holds. 
The subspaces $\frak{u}=\frak{q}^{\perp {\frak{g}}}$ and
 $\frak{u}'={\frak{q}'}^{\perp {\frak{g}'}}$
 are the nilradicals of $\frak{q}$ and $\frak{q}'$, respectively.
Because $\frak{q}$ and $\frak{q}'$ are $\theta$-stable, we have
$(\frak{q}\cap\frak{p})^{\perp {\frak{p}}}=\frak{u}\cap\frak{p}$ and 
$(\frak{q}'\cap\frak{p}')^{\perp {\frak{p}'}}=\frak{u}'\cap\frak{p}'$.
In view of
 Lemma~\ref{proj} and $\frak{q}\cap \frak{p}'=\frak{q}'\cap \frak{p}'$,
 we get
\begin{align}
\label{proju}
\textstyle\pr_{\frak{g}\to\frak{g}'}(\frak{u}\cap\frak{p})
=\textstyle\pr_{\frak{g}\to\frak{g}'}((\frak{q}\cap\frak{p})^{\perp {\frak{p}}})
=(\frak{q}\cap\frak{p}')^{\perp {\frak{p}'}}
=(\frak{q}'\cap\frak{p}')^{\perp {\frak{p}'}}
=\frak{u}'\cap \frak{p}'.\nonumber
\end{align}
The right side is contained in ${\cal{N}}_{\frak{g}'}$.
This shows $({\rm{iv}})$.

Assume that $({\rm{iv}})$ holds.
As we have seen above, 
\[\textstyle\pr_{\frak{g}\to\frak{g}'}((\frak{q}\cap\frak{p})^{\perp {\frak{p}}})
=(\frak{q}\cap\frak{p}')^{\perp {\frak{p}'}}.\]
Since the vector space $(\frak{q}\cap\frak{p}')^{\perp {\frak{p}'}}$
 is contained in the nilpotent cone of $\frak{g}'$,
 the bilinear form $\langle \cdot, \cdot \rangle$ is zero on
 $(\frak{q}\cap\frak{p}')^{\perp {\frak{p}'}}$ and hence
 $(\frak{q}\cap \frak{p}')^{\perp {\frak{p}'}}\subset\frak{q}\cap \frak{p}'$.
Then it follows that
 $N_{\frak{k}'}(\frak{q}\cap\frak{p}')
=[(\frak{q}\cap \frak{p}'), 
(\frak{q}\cap \frak{p}')^{\perp {\frak{p}'}}]^{\perp {\frak{k}'}}$.
Indeed, for $x\in\frak{k}'$, 
\begin{align*}
x\in [(\frak{q}\cap \frak{p}'), 
(\frak{q}\cap \frak{p}')^{\perp {\frak{p}'}}]^{\perp {\frak{k}'}}
&\Leftrightarrow
\langle x,[(\frak{q}\cap \frak{p}'), 
(\frak{q}\cap \frak{p}')^{\perp {\frak{p}'}}]\rangle=\{0\} \\
&\Leftrightarrow
\langle [x, (\frak{q}\cap \frak{p}')], 
(\frak{q}\cap \frak{p}')^{\perp {\frak{p}'}}\rangle=\{0\} \\
&\Leftrightarrow
[x, (\frak{q}\cap \frak{p}')] \in \frak{q}\cap \frak{p}' \\
&\Leftrightarrow
x\in N_{\frak{k}'}(\frak{q}\cap \frak{p}').
\end{align*}
Put 
$\frak{q}':=N_{\frak{k}'}(\frak{q}\cap \frak{p}')
+(\frak{q}\cap\frak{p}')$.
Then 
\begin{align*}
{\frak{q}'}^{\perp {\frak{g}'}}
=N_{\frak{k}'}(\frak{q}\cap \frak{p}')^{\perp {\frak{k}'}}+
(\frak{q}\cap \frak{p}')^{\perp {\frak{p}'}}
=[(\frak{q}\cap \frak{p}'), 
(\frak{q}\cap \frak{p}')^{\perp {\frak{p}'}}]+
(\frak{q}\cap \frak{p}')^{\perp {\frak{p}'}}.
\end{align*}
Since 
$[(\frak{q}\cap \frak{p}'), 
(\frak{q}\cap \frak{p}')^{\perp {\frak{p}'}}]
\subset [(\frak{q}\cap \frak{p}'), 
(\frak{q}\cap \frak{p}')]
\subset N_{\frak{k}'}(\frak{q}\cap \frak{p}'),$
 we see that 
${\frak{q}'}^{\perp {\frak{g}'}}\subset \frak{q}'$.
We therefore have 
$\langle x, y\rangle=0$ for $x,y\in{{\frak{q}'}^{\perp {\frak{g}'}}}$.
Moreover, ${\frak{q}'}^{\perp {\frak{g}'}}$ is a subalgebra of $\frak{g}'$
 because 
\begin{align*}
\langle[{\frak{q}'}^{\perp {\frak{g}'}},{\frak{q}'}^{\perp {\frak{g}'}}],
\frak{q}'\rangle
=\langle {\frak{q}'}^{\perp {\frak{g}'}}, [{\frak{q}'}^{\perp {\frak{g}'}},
\frak{q}']\rangle 
\subset \langle {\frak{q}'}^{\perp {\frak{g}'}}, \frak{q}'\rangle=\{0\}.
\end{align*}
As a consequence, ${\frak{q}'}^{\perp {\frak{g}'}}$
 is a solvable Lie algebra and hence contained in some Borel subalgebra
 $\frak{b}'$ of $\frak{g}'$.
Write $\frak{n}'$ for the nilradical of $\frak{b}'$ so
 $\frak{n}'=\frak{b}'^{\perp {\frak{g}'}}$.
Let $M:=N_{K'}(\frak{q}\cap \frak{p}')$ be the normalizer of
 $\frak{q}\cap \frak{p}'$, which is an algebraic subgroup of $K'$.
Then $M$ has a Levi decomposition with reductive part $M_R$ and
 unipotent part $M_U$ (see \cite[\S VIII.4]{Hoch}).
If we denote by $\frak{m}_R$ and $\frak{m}_U$ the Lie algebras
 of $M_R$ and $M_U$, respectively, then
 the bilinear form $\langle\cdot,\cdot\rangle$ is non-degenerate
 on $\frak{m}_R$ and zero on $\frak{m}_U$.
We then conclude that the nilradical of
 $N_{\frak{k}'}(\frak{q}\cap \frak{p}')$ equals the radical of
 $N_{\frak{k}'}(\frak{q}\cap \frak{p}')$ with respect to the bilinear form.
As a result, 
$[(\frak{q}\cap \frak{p}'),
 (\frak{q}\cap \frak{p}')^{\perp {\frak{p}'}}]
 =N_{\frak{k}'}(\frak{q}\cap \frak{p}')^{\perp {\frak{k}'}}$
is the nilradical of $N_{\frak{k}'}(\frak{q}\cap \frak{p}')$
 and hence $[(\frak{q}\cap \frak{p}'),
 (\frak{q}\cap \frak{p}')^{\perp {\frak{p}'}}]\subset \frak{n}'$.
Since $(\frak{q}\cap \frak{p}')^{\perp {\frak{p}'}}\subset 
 {\cal N}_{\frak{g}'}\cap \frak{b}'=\frak{n}'$,
 it follows that 
$\frak{q}'^{\perp {\frak{g}'}}
 =[(\frak{q}\cap \frak{p}'), 
(\frak{q}\cap \frak{p}')^{\perp {\frak{p}'}}]+
(\frak{q}\cap \frak{p}')^{\perp {\frak{p}'}}\subset \frak{n}'$. 
Hence we see that
 $\frak{q}'\supset \frak{n}'^{\perp {\frak{g}'}}=\frak{b}'$
 and $\frak{q}'$ is a parabolic subalgebra of $\frak{g}'$,
 showing (iii).
\end{proof}

Retain the notation and the assumption of Theorem~\ref{parab} and
 suppose that the equivalent conditions in Theorem~\ref{parab}
 are satisfied.
Let ${\cal Q}$ be the set of all $\theta$-stable parabolic subalgebras
 $\frak{q}'_i$ of $\frak{g}'$ such that
 $\frak{q}'_i\cap \frak{p}' =\frak{q}\cap \frak{p}'$.
Then the parabolic subalgebra
$\frak{q}'=N_{\frak{k}'}(\frak{q}\cap\frak{p}')+(\frak{q}\cap\frak{p}')$
 given in Theorem~\ref{parab}
 is a unique maximal element of ${\cal Q}$.

On the other hand,
 a minimal element $\frak{q}''$ of ${\cal Q}$ is constructed as follows.
For the parabolic subalgebra $\frak{q}'$ defined above,
 put $\frak{l}'=\frak{q}'\cap \ovl{\frak{q}'}$, which is a Levi part of
 $\frak{q}'$.
The $\theta$-stable reductive subalgebra $\frak{l}'$ decomposes as
\[\frak{l}'=\bigoplus_{i\in I}\frak{l}'_i \oplus \frak{z}(\frak{l}'),\]
 where $\frak{l}'_i$ are simple Lie algebras and $\frak{z}(\frak{l}')$
 is the center of $\frak{l}'$. 
Put $I_c:=\{i\in I: \frak{l}'_i\subset \frak{k}'\}$ and define
\[\frak{l}'_c:= \bigoplus_{i\in I_c}\frak{l}'_i
 \oplus (\frak{z}(\frak{l}')\cap \frak{k}'),
\quad \frak{l}'_n:= \bigoplus_{i \not\in I_c}\frak{l}'_i
 \oplus (\frak{z}(\frak{l}')\cap \frak{p}').
\]
Then we have
\[\frak{l}'=\frak{l}'_c\oplus \frak{l}'_n, \quad
 \frak{l}'_n= [(\frak{l}'\cap \frak{p}'),(\frak{l}'\cap \frak{p}')]
 +\frak{l}'\cap \frak{p}', \quad
 \frak{l}'_c\subset \frak{k}'.\]
Take a Borel subalgebra $\frak{b}(\frak{l}'_c)$ of $\frak{l}'_c$ and
 define
\begin{align}
\label{defq''}
\frak{q}'':=\frak{b}(\frak{l}'_c)\oplus\frak{l}'_n\oplus\frak{u}'.
\end{align}

We claim that $\frak{q}''$ is a minimal element of ${\cal Q}$ and
 every minimal element is obtained in this way.
Indeed, since any element $\frak{q}'_i$
 of ${\cal Q}$ is contained in $\frak{q}'$, 
 the parabolic subalgebra $\frak{q}'_i$ 
 decomposes as $(\frak{q}'_i\cap \frak{l}')\oplus \frak{u}'$.
The condition $\frak{q}'_i\cap \frak{p}'=\frak{q}\cap \frak{p}'$
 implies that $\frak{q}'_i\supset\frak{l}'\cap \frak{p}'$
 and hence
 $\frak{q}'_i\supset \frak{l}'_n$.
As a consequence, the set ${\cal Q}$ consists of the Lie algebras
$\frak{q}(\frak{l}'_c)\oplus\frak{l}'_n\oplus\frak{u}'$
 for parabolic subalgebras $\frak{q}(\frak{l}'_c)$ of $\frak{l}'_c$.
Our claim follows from this.
In particular, a minimal element of ${\cal Q}$ is unique up
 to inner automorphisms of $\frak{l}'_c$.

We note here some observations on Lie algebras
 for later use.

\begin{lem}
\label{unisub}
Retain the notation and the assumption above.
Then 
\[\frak{q}\cap \frak{g}'=(\frak{q}\cap \frak{l}'_c)
 \oplus \frak{l}'_n \oplus \frak{u}',\]
and
\[[(\frak{l}'_n+\frak{u}'),\frak{g}]\subset \frak{q}+\frak{g}'.\]
\end{lem}

\begin{proof}
From $\frak{q}\cap \frak{k}'\subset N_{\frak{k}'}(\frak{q}\cap \frak{p}')$
 and $\frak{q}\cap \frak{p}'=\frak{q}'\cap \frak{p}'$, we have
 $\frak{q}\cap \frak{g}'\subset \frak{q}'$.
From the proof of Theorem~\ref{parab}, 
we have 
\begin{align*}
\frak{u}'=\frak{q}'^{\perp \frak{g}'}
&=[(\frak{q}\cap \frak{p}'),(\frak{q}\cap \frak{p}')^{\perp \frak{p}'}]
+(\frak{q}\cap \frak{p}')^{\perp \frak{p}'} \\
&\subset
[(\frak{q}\cap \frak{p}'),(\frak{q}\cap \frak{p}')]
+(\frak{q}\cap \frak{p}')
\subset \frak{q}\cap \frak{g}'.
\end{align*}
Moreover, $\frak{l}'_n 
 =[(\frak{l}'\cap \frak{p}'),(\frak{l}'\cap \frak{p}')]
 +(\frak{l}'\cap \frak{p}')$ and
 $\frak{l}'\cap \frak{p}'\subset \frak{q}'\cap \frak{p}'
 = \frak{q}\cap \frak{p}'$
imply that
 $\frak{l}'_n\subset \frak{q}\cap \frak{g}'$.
Hence
 $\frak{q}\cap \frak{g}'$ decomposes as
 $\frak{q}\cap \frak{g}'=(\frak{q}\cap \frak{l}'_c)
 \oplus \frak{l}'_n \oplus \frak{u}'$.

For the second assertion, we see that
 $[(\frak{q}\cap \frak{p}'),\frak{g}]\subset\frak{q}+\frak{g}'$.
Indeed, the assumption
 $(\frak{q}\cap \frak{k})+\frak{k}'=\frak{k}$ implies
 that
\[ [(\frak{q}\cap \frak{p}'),\frak{k}]
=[(\frak{q}\cap \frak{p}'),(\frak{q}\cap \frak{k})]
+[(\frak{q}\cap \frak{p}'),\frak{k}']
\subset \frak{q}+\frak{g}'\]
and $[(\frak{q}\cap \frak{p}'),\frak{p}]\subset
\frak{k}\subset\frak{q}+\frak{g}'$.
Hence $[(\frak{q}\cap \frak{p}'), \frak{g}]\subset \frak{q}+\frak{g}'$.
Then the inclusion $[\frak{u}',\frak{g}]\subset \frak{q}+\frak{g}'$
 follows from
$\frak{u}'\subset
[(\frak{q}\cap \frak{p}'),(\frak{q}\cap \frak{p}')]
+(\frak{q}\cap \frak{p}')$,
 and the inclusion $[\frak{l}'_n,\frak{g}]\subset \frak{q}+\frak{g}'$
 follows from
 $\frak{l}'_n 
 =[(\frak{l}'\cap \frak{p}'), (\frak{l}'\cap \frak{p}')]
 +(\frak{l}'\cap \frak{p}')$.
\end{proof}


\section{Upper Bound on Branching Law}
\label{sec:branch}
We retain the notation of the previous section.

\begin{prop}
\label{square}
Suppose that the equivalent conditions in Theorem~\ref{parab} hold
 for a $\sigma$-open $\theta$-stable parabolic
 subalgebra $\frak{q}$ of $\frak{g}$.
Define $\frak{q}'$ as in Theorem~\ref{parab}
 and define $\ovl{Q'}$ as the parabolic subgroup of $G'$
 with Lie algebra $\ovl{\frak{q}'}$.
Then $\ovl{Q}\cap G'\subset \ovl{Q'}$.
Moreover, the following is a Cartesian square.
\begin{align*}
\xymatrix{
K'/(\ovl{Q}\cap K') \ar[r]^*+{i^o} \ar[d]
& G'/(\ovl{Q}\cap G') \ar[d]^*+{\pi} \\
K'/(\ovl{Q'}\cap K') \ar[r]^<<<<<<*+{i'}   & G'/\ovl{Q'}
}
\end{align*}
In particular, $i^o$ is a closed immersion.
\end{prop}

\begin{proof}
Let $g\in \ovl{Q}\cap G'$.
To see $g\in \ovl{Q'}$, it enough to show that $\Ad(g)$ normalizes
 $\ovl{\frak{q}'}$ because $\ovl{Q'}$ is self-normalizing.
By Lemma~\ref{unisub},
 $\ovl{\frak{u}'}\subset \Bar{\frak{q}}\cap \frak{g}' \subset \ovl{\frak{q}'}$.
Therefore,
 $\Ad(g)(\Bar{\frak{q}}\cap \frak{g}')=\Bar{\frak{q}}\cap \frak{g}'$
 implies that
 $\Ad(g)\ovl{\frak{u}'}\subset \ovl{\frak{q}'}$.
Then $\Ad(g)\ovl{\frak{q}'}\subset \ovl{\frak{q}'}$ follows from
 the lemma below:
\begin{lem}
Let $\frak{g}$ be a reductive Lie algebra and
 $\frak{q}$ a parabolic subalgebra.
If $\phi(\frak{u})\subset \frak{q}$ for the nilradical
 $\frak{u}$ of $\frak{q}$ and
 an inner automorphism $\phi\in \Int(\frak{g})$,
 then $\phi(\frak{q})=\frak{q}$.
\end{lem}
\begin{proof}
There exists a Cartan subalgebra $\frak{h}$ of $\frak{g}$ contained in
 both $\frak{q}$ and $\phi(\frak{q})$.
Our assumption amounts to the inclusion of the sets of $\frak{h}$-roots
 $\Delta(\phi(\frak{u}), \frak{h})\subset \Delta(\frak{q},\frak{h})$.
Write $\frak{l}$ for the Levi part of $\frak{q}$ containing $\frak{h}$.
Then 
\[
\Delta(\phi(\frak{q}),\frak{h})\cap \Delta(\frak{q},\frak{h})
=\Delta(\phi(\frak{u}),\frak{h})
\cup (\Delta(\phi(\frak{l}),\frak{h}) \cap \Delta(\frak{q},\frak{h})).
\]
As a result, $\phi(\frak{q})\cap \frak{q}$ is a parabolic subalgebra
 of $\frak{g}$.
In particular, $\phi(\frak{q})$ and $\frak{q}$ have a common
 Borel subalgebra.
Since $\phi$ is inner, this implies that $\phi(\frak{q})=\frak{q}$.
\end{proof}

Returning to the proof of Proposition~\ref{square},
 we now prove that the diagram is a Cartesian square.
This is equivalent to
 that $\ovl{Q'}=(\ovl{Q'}\cap K')\cdot(\ovl{Q}\cap G')$.
The inclusion $\ovl{Q'}\supset(\ovl{Q'}\cap K')\cdot(\ovl{Q}\cap G')$
 follows from $\ovl{Q'}\supset (\ovl{Q}\cap G')$.
Since $\ovl{Q'}$ is connected and $\theta$-stable, it is 
 generated by $\ovl{Q'}\cap K'$ and
 $\exp (\ovl{\frak{q}'}\cap \frak{p}')$ as a group.
For $k\in \ovl{Q'}\cap K'$ and $x\in \ovl{\frak{q}'}\cap \frak{p}'$,
 we have $\exp(x)k=k \exp (\Ad(k^{-1})x)$ and
 $\Ad(k^{-1})x\in \ovl{\frak{q}'}\cap \frak{p}'$.
Using this equation iteratively, we can write any element
 of $\ovl{Q'}$ as $k \exp(x_1)\cdots \exp(x_n)$ for
 $k\in \ovl{Q'}\cap K'$ and
 $x_1,\dots,x_n\in \ovl{\frak{q}'}\cap \frak{p}'$.
Then $\ovl{\frak{q}'}\cap \frak{p}'=\ovl{\frak{q}}\cap \frak{p}'$
 implies that $\exp(x_1)\cdots\exp(x_n)\in \ovl{Q}\cap G'$.
Hence $\ovl{Q'}\subset(\ovl{Q'}\cap K')\cdot(\ovl{Q}\cap G')$ as required.
\end{proof}

Now we consider the restriction
 $A_\frak{q}(\lambda)|_{(\frak{g'},K')}$.
We assume that $\lambda$ is linear, so
 the $(\frak{l},L\cap K)$-action on $\bb{C}_\lambda$ can 
 be uniquely extended to an $L$-action or a $\ovl{Q}$-action.

Define
\[V^p:=\bigwedge^{\rm top}(\frak{g}/(\Bar{\frak{q}}+\frak{g}'))
 \otimes S^p(\frak{g}/(\Bar{\frak{q}}+\frak{g}'))\]
regarded as a $(\ovl{Q}\cap G')$-module by the adjoint action 
and define
\[W^p:= \operatorname{Ind}_{\ovl{Q}\cap G'}^{\ovl{Q'}}
 (\bb{C}_\lambda|_{\ovl{Q}\cap G'}\otimes V^p).
\]
By Lemma~\ref{unisub}, 
 the unipotent radical ${\ovl{U'}}$ of ${\ovl{Q'}}$
 is contained in $\ovl{Q}\cap G'$ 
 and ${\ovl{U'}}$ acts trivially on
 $\bb{C}_\lambda|_{\ovl{Q}\cap G'}\otimes V^p$.
Therefore, ${\ovl{U'}}$ acts trivially on $W^p$. 
Then
 $W^p$ is written as a direct sum of
 irreducible finite-dimensional ${\ovl{Q'}}$-modules
 and ${\ovl{U'}}$ acts trivially on all the irreducible
 components.
As an $L'$-module, we have
\[
W^p \simeq \operatorname{Ind}_{\ovl{Q}\cap L'}^{L'}
 (\bb{C}_\lambda|_{\ovl{Q}\cap L'}\otimes V^p).
\]

\begin{thm}
\label{branchup1}
Let $(G_0,G'_0)$ be a symmetric pair of connected real
 linear reductive Lie groups
 defined by an involution $\sigma$.
Let $\frak{q}$ be a $\sigma$-open $\theta$-stable parabolic subalgebra
 of $\frak{g}$.
Suppose that $A_\frak{q}(\lambda)$ is nonzero and discretely decomposable
 as a $(\frak{g}',K')$-module with $\lambda$ linear, unitary,
 and in the weakly fair range.
Define 
\begin{align*}
\frak{q}':=N_{\frak{k}'}(\frak{q}\cap \frak{p}')+(\frak{q}\cap \frak{p}'),
\end{align*}
and 
\[
W^p:= \operatorname{Ind}_{\ovl{Q}\cap G'}^{\ovl{Q'}}
 \left(\bb{C}_\lambda|_{\ovl{Q}\cap G'}\otimes 
 \bigwedge^{\rm top}(\frak{g}/(\Bar{\frak{q}}+\frak{g}'))
 \otimes S^p(\frak{g}/(\Bar{\frak{q}}+\frak{g}'))\right).
\]
Then there exists an injective homomorphism of
 $(\frak{g}',K')$-modules
\begin{align}
\label{eqn:brinj1}
A_\frak{q}(\lambda)\to
 \bigoplus_{p=0}^{\infty} 
 (\Pi_{L'\cap K'}^{K'})_{s'}
 \bigl(U(\frak{g}')\otimes_{U(\ovl{\frak{q}'})}
 (W^p \otimes \bb{C}_{2\rho(\frak{u}')})\bigr)
\end{align}
for $s'= \dim (\frak{u}'\cap \frak{k}')$.
\end{thm}

\begin{proof}
Suppose that $A_\frak{q}(\lambda)$ is
 nonzero and
 discretely decomposable
 as a $(\frak{g}',K')$-module with $\lambda$ linear, unitary,
 and in the weakly
 fair range.
Let $\ovl{Q}$, $G'$, and $K'$
 be the connected subgroups of $G$
 with Lie algebras $\Bar{\frak{q}}$, $\frak{g}'$ and $\frak{k}'$, respectively.
We set 
\begin{align*}
&X=G/\ovl{Q},\quad 
X^o=G'/(\ovl{Q}\cap G'),\quad \\
&Y=K/(\ovl{Q}\cap K), \quad 
Y^o=K'/(\ovl{Q}\cap K'),\\
&\xymatrix{
Y \ar[r]^*+{i}  & X \\
Y^o \ar[u]^*+{j_K}   \ar[r]^*+{i^o} & X^o \ar[u]_*+{j}
}
\end{align*}
where the maps $i^o,i, j$, and $j_K$ are the inclusion maps.
The map $j_K$ is an open immersion because $\frak{q}$ is
 $\sigma$-open.
By Lemma~\ref{square}, $i^o$ is a closed immersion and hence
 $i(Y)\cap j(X^o)=i(j_K(Y^o))$.

Let ${\cal L}_{\lambda,X}$ be the 
 ${\cal O}_X$-module
 associated with the $\ovl{Q}$-module $\bb{C}_\lambda$ 
 as in Section~\ref{sec:homog}.
Then Theorem~\ref{locZuc} says 
 $\Gamma(X,i_+{\cal O}_{Y}
 \otimes_{{\cal O}_X}{\cal L}_{\lambda,X})$
 is isomorphic to $A_\frak{q}(\lambda)$
 as a $(\frak{g},K)$-module.
We see that
\[j^{-1}i_+{\cal O}_{Y}
\simeq j^{-1}(j\circ i^o)_+ {\cal O}_{Y^o}
\simeq j^{-1}j_+ (i^o_+ {\cal O}_{Y^o}).\]

Let $\{F_p {\cal D}_X\}_{p\geq 0}$
 be the filtration by normal degree with respect to $j$.
This induces a filtration
 $\{F_p j^{-1}i_+{\cal O}_{Y}\}$ on
 $j^{-1} i_+ {\cal O}_{Y}$ and a filtration
$\{F_p j^{-1}(i_+{\cal O}_{Y} \otimes_{{\cal O}_X}
  {\cal L}_{\lambda,X})\}$ on
 $j^{-1} (i_+ {\cal O}_{Y}\otimes_{{\cal O}_X} {\cal L}_{\lambda,X})$. 
Applying Lemma~\ref{fil} for ${\cal M}=i^o_+ {\cal O}_{Y^o}$,
 we have isomorphisms of ${\cal O}_{X^o}$-modules
\begin{align*}
F_p j^{-1}i_+{\cal O}_{Y}
/F_{p-1} j^{-1}i_+{\cal O}_{Y}
&\simeq
F_p j^{-1}j_+ (i^o_+ {\cal O}_{Y^o})
/F_{p-1} j^{-1}j_+ (i^o_+ {\cal O}_{Y^o})\\
&\simeq
(i^o_+ {\cal O}_{Y^o})\otimes_{{\cal O}_{X^o}} 
\Omega_{X/X^o}\spcheck\otimes_{{\cal O}_{X^o}}
j^{-1}({\cal I}_{X^o}^p/{\cal I}_{X^o}^{p+1})\spcheck,
\end{align*}
which commute with the actions of $\frak{g}'$ and $K'$.
The $G'$-equivariant ${\cal O}_{X^o}$-module
$\Omega_{X/X^o}\spcheck\otimes_{{\cal O}_{X^o}}
j^{-1}({\cal I}_{X^o}^p/{\cal I}_{X^o}^{p+1})\spcheck$
 is isomorphic to the ${\cal O}_{X^o}$-module ${\cal V}_{X^o}^p$
 associated with the $(\ovl{Q'}\cap G')$-module 
\[
V^p:=\bigwedge^{\rm top}(\frak{g}/(\Bar{\frak{q}}+\frak{g}'))
\otimes S^p(\frak{g}/(\Bar{\frak{q}}+\frak{g}')).
\]
We write ${\cal L}_{\lambda,X^o}$
 for the ${\cal O}_{X^o}$-module associated with
 $\bb{C}_\lambda|_{\ovl{Q}\cap G'}$.
Then $j^*{\cal L}_{\lambda,X}\simeq {\cal L}_{\lambda,X^o}$.
As a result, we get an isomorphism
\begin{align}
\label{eqn:fil}
&F_p j^{-1}(i_+{\cal O}_Y\otimes_{{\cal O}_X} {\cal L}_{\lambda,X})
/F_{p-1} j^{-1}(i_+{\cal O}_Y
 \otimes_{{\cal O}_X} {\cal L}_{\lambda,X})\\ \nonumber
\simeq\ &i^o_+{\cal O}_{Y^o} \otimes_{{\cal O}_{X^o}} 
{\cal L}_{\lambda,X^o}\otimes_{{\cal O}_{X^o}} {\cal V}_{X^o}^p.
\end{align}

Since any section
 $m\in \Gamma(X, i_+{\cal O}_Y\otimes_{{\cal O}_X} {\cal L}_{\lambda,X})$
 is $K$-finite, 
 the support of $m$ is $Y$ unless $m=0$.
Therefore, the restriction map
\[r:
\Gamma(X, i_+{\cal O}_Y\otimes_{{\cal O}_X} {\cal L}_{\lambda,X})
\to \Gamma(X^o, j^{-1}(i_+{\cal O}_Y\otimes_{{\cal O}_X}
 {\cal L}_{\lambda,X}))
\]
is injective.
Define the filtration $\{F_p A_\frak{q}(\lambda)\}$
 of the $(\frak{g}',K')$-module $A_\frak{q}(\lambda)$
 by 
\[F_p A_\frak{q}(\lambda):= 
r^{-1}\Gamma(X^o, F_p j^{-1}(i_+{\cal O}_Y\otimes_{{\cal O}_X}
 {\cal L}_{\lambda,X}))\]
for 
\[r:
A_\frak{q}(\lambda)\simeq 
\Gamma(X, i_+{\cal O}_Y\otimes_{{\cal O}_X} {\cal L}_{\lambda,X})
\to \Gamma(X^o, j^{-1}(i_+{\cal O}_Y\otimes_{{\cal O}_X}
 {\cal L}_{\lambda,X})).
\]
The induced map
\begin{align*}
&F_p A_\frak{q}(\lambda)/F_{p-1} A_\frak{q}(\lambda) \\
\to\ & \Gamma(X^o, F_p j^{-1}(i_+{\cal O}_Y\otimes_{{\cal O}_X}
 {\cal L}_{\lambda,X}))
\,/\,\Gamma(X^o, F_{p-1} j^{-1}(i_+{\cal O}_Y\otimes_{{\cal O}_X}
 {\cal L}_{\lambda,X})).
\end{align*}
is injective.
The unitarizability and the discrete decomposability
 of $A_\frak{q}(\lambda)$
 imply that there exists an isomorphism of
 the $(\frak{g}',K')$-modules
\[A_\frak{q}(\lambda)\simeq
 \bigoplus_{p=0}^{\infty}
 F_p A_\frak{q}(\lambda)/F_{p-1} A_\frak{q}(\lambda).
\]
Consequently, we obtain injective maps of $(\frak{g}',K')$-modules
\begin{align}
\label{eqn:gr}
 &A_\frak{q}(\lambda)
\simeq\bigoplus_{p=0}^{\infty}
 F_p A_\frak{q}(\lambda)/F_{p-1} A_\frak{q}(\lambda)\\ \nonumber
\to\ &\bigoplus_{p=0}^{\infty}
 \Gamma(X^o, F_p j^{-1}(i_+{\cal O}_Y\otimes_{{\cal O}_X}
  {\cal L}_{\lambda,X}))
\,/\,\Gamma(X^o, F_{p-1} j^{-1}(i_+{\cal O}_Y\otimes_{{\cal O}_X}
 {\cal L}_{\lambda,X}))
 \\ \nonumber
\to\ &\bigoplus_{p=0}^{\infty}
 \Gamma\bigl(X^o, F_p j^{-1}(i_+{\cal O}_Y\otimes_{{\cal O}_X}
 {\cal L}_{\lambda,X})
 /F_{p-1} j^{-1}(i_+{\cal O}_Y\otimes_{{\cal O}_X} {\cal L}_{\lambda,X})\bigr).
\end{align}
The injectivity of the last map follows from
 the left exactness of the functor $\Gamma(X^o,\cdot\,)$.

We set 
\begin{align*}
&X'=G'/\ovl{Q'}, \quad 
Y'=K'/(\ovl{Q'}\cap K'), \\
&\xymatrix{
Y^o \ar[r]^*+{i^o} \ar[d]_*+{\pi_K} 
& X^o \ar[d]^*+{\pi}\\
Y'   \ar[r]^*+{i'}   & X'
}
\end{align*}
where the maps in the commutative diagram
 are defined canonically.
Since the diagram is a Cartesian square by Lemma~\ref{square}
 and $\pi, \pi_K$ are smooth morphisms,
 the base change formula
 gives isomorphisms of
 ${\cal{D}}_{X^o}$-modules
\[
i^o_+{\cal O}_{Y^o}\simeq i^o_+ \pi_K^*{\cal O}_{Y'}
\simeq \pi^*i'_+{\cal O}_{Y'}.
\]
Then the projection formula gives the 
following isomorphisms of ${\cal O}_{X'}$-modules 
\begin{align*}
\pi_*(i^o_+{\cal O}_{Y^o} \otimes_{{\cal O}_{X^o}}
{\cal L}_{\lambda, X^o}
\otimes_{{\cal O}_{X^o}} {\cal V}^p_{X^o}) 
&\simeq 
 \pi_*(\pi^*i'_+{\cal O}_{Y'}\otimes_{{\cal O}_{X^o}}
 {\cal L}_{\lambda, X^o}
\otimes_{{\cal O}_{X^o}} {\cal V}^p_{X^o}) \\
&\simeq 
 i'_+{\cal O}_{Y'}\otimes_{{\cal O}_{X'}}
\pi_*({\cal L}_{\lambda,X^o}
\otimes_{{\cal O}_{X^o}} {\cal V}^p_{X^o}),
\end{align*}
which commute with the actions of $\frak{g}'$ and $K'$.
Put $S:= \ovl{Q'}/(\ovl{Q}\cap G')$.
By Lemma~\ref{push},
 $\pi_*({\cal L}_{\lambda,X^o}\otimes_{{\cal O}_{X^o}} {\cal V}^p_{X^o})$
 is isomorphic to
 the ${\cal O}_{X'}$-module ${\cal W}^p_{X'}$ 
 associated with
 the $\ovl{Q'}$-module
 $W^p:=\Gamma(S, {\cal V}^p_S)$, or equivalently
\[
W^p:=\operatorname{Ind}_{\ovl{Q}\cap G'}^{\ovl{Q'}}
 \left(\bb{C}_\lambda|_{\ovl{Q}\cap G'}\otimes
 \bigwedge^{\rm top}(\frak{g}/(\Bar{\frak{q}}+\frak{g}'))
 \otimes
 S^p(\frak{g}/(\Bar{\frak{q}}+\frak{g}'))\right).
\]
Therefore, 
\begin{align}
\label{eqn:push}
&\Gamma(X^o, i^o_+{\cal O}_{Y^o}\otimes_{{\cal O}_{X^o}}
 {\cal L}_{\lambda,X^o} 
 \otimes_{{\cal O}_{X^o}} {\cal V}^p_{X^o})\\ \nonumber
\simeq\ &
\Gamma(X', i'_+{\cal O}_{Y'}\otimes_{{\cal O}_{X'}}
 \pi_*({\cal L}_{\lambda,X^o} 
 \otimes_{{\cal O}_{X^o}} {\cal V}^p_{X^o}))\\ \nonumber
\simeq\ &
\Gamma(X', i'_+{\cal O}_{Y'}\otimes_{{\cal O}_{X'}}
 {\cal W}^p_{X'}).
\end{align}

Combining
\eqref{eqn:fil}, \eqref{eqn:gr}, and \eqref{eqn:push}, 
we obtain an injective $(\frak{g}',K')$-homomorphism
\begin{align}
\label{eqn:injloc}
A_\frak{q}(\lambda)\to
\bigoplus_{p=0}^\infty
\Gamma(X', i'_+{\cal O}_{Y'}\otimes_{{\cal O}_{X'}}
 {\cal W}^p_{X'}).
\end{align}
Finally, Theorem~\ref{locZuc} gives an isomorphism
\[
\Gamma(X', i'_+{\cal O}_{Y'}\otimes_{{\cal O}_{X'}}
 {\cal W}^p_{X'})
\simeq 
(\Pi_{L'\cap K'}^{K'})_{s'}
(U(\frak{g}')\otimes_{U(\ovl{\frak{q}'})} 
(W^p\otimes\bb{C}_{2\rho(\frak{u}')})),
\]
so we have completed the proof.
\end{proof}

Let $\frak{q}''$ be the $\theta$-stable parabolic subalgebra
 of $\frak{g}'$ defined by \eqref{defq''}.
In what follows, we show that
 the right side of \eqref{eqn:brinj1} can be written as
 the direct sum of $(\frak{g}',K')$-modules $A_{\frak{q}''}(\lambda')$.

Let $L''_0:=N_{G'_0}(\ovl{\frak{q}''})$ be the normalizer of 
 $\ovl{\frak{q}''}$ in $G'_0$.
The complexified Lie algebra $\frak{l}''$ decomposes as
 $\frak{l}''=(\frak{l}''\cap \frak{l}'_c)\oplus\frak{l}'_n$.
Then $\frak{h}'_c:=\frak{l}''\cap \frak{l}'_c$
 is a Cartan subalgebra of $\frak{l}'_c$.
The center $\frak{z}(\frak{l}'')$ of $\frak{l}''$ decomposes as
\[\frak{z}(\frak{l}'')=
\frak{h}'_c\oplus
(\frak{z}(\frak{l}'')\cap \frak{l}'_n).
\]
Write $\lambda'=\lambda'_c+\lambda'_n$ for the
 corresponding decomposition of $\lambda'\in\frak{z}(\frak{l}'')^*$.
We take $\Delta(\frak{b}(\frak{l}'_c),\frak{h}'_c)$
 as a positive root system of 
 $\Delta(\frak{l}'_c,\frak{h}'_c)$.
If $\lambda'_c\in(\frak{h}'_c)^*$ is dominant integral
 for $\Delta(\frak{b}(\frak{l}'_c),\frak{h}'_c)$,
 write $F(\lambda'_c)$ for the irreducible
 finite-dimensional representation
 of $\frak{l}'_c$ with highest weight $\lambda'_c$.

Let $\Lambda$ be the set consisting of
 $\lambda'=\lambda'_c+\lambda'_n\in\frak{z}(\frak{l}'')^*$
 such that
\begin{itemize}
\item
$\lambda'$ is linear,
\item
$\lambda'_c$
 is dominant for
 $\Delta(\frak{b}(\frak{l}'_c),\frak{h}'_c)$, and 
\item
$\lambda'_n=0$.
\end{itemize}
For $\lambda'\in \Lambda$, define
 the representation $F(\lambda')$
 of $\frak{l}'=\frak{l}'_c\oplus\frak{l}'_n$ by
 the exterior tensor product of $F(\lambda'_c)$ and the trivial representation
 of $\frak{l}'_n$:
\[
F(\lambda'):=F(\lambda'_c)\boxtimes \bb{C}.
\]
Since $\lambda'$ is linear, 
$F(\lambda')$ lifts to a representation of $L'$.
Define
\begin{align}
\label{defmult}
m(\lambda', p):=
\dim {\Hom}_{\ovl{Q}\cap L'}
\left( F(\lambda'),\, \bb{C}_\lambda|_{\ovl{Q}\cap G'}\otimes
\bigwedge^{\rm top}(\frak{g}/(\Bar{\frak{q}}+\frak{g}'))
\otimes S^p(\frak{g}/(\Bar{\frak{q}}+\frak{g}'))\right).
\end{align}

\begin{thm}
\label{branchup2}
Let the notation and the assumption be as in Theorem~\ref{branchup1}.
Define $\frak{q}''$ as in \eqref{defq''} and define
 $\Lambda$, $m(\lambda',p)$ as above.
Then there exists an injective homomorphism of
 $(\frak{g}',K')$-modules
\begin{align}
\label{eqn:brinj2}
A_\frak{q}(\lambda)\to
 \bigoplus_{p=0}^{\infty} 
 \bigoplus_{\lambda'\in\Lambda}
 A_{\frak{q}''}(\lambda')^{\oplus m(\lambda'\!,\,p)}.
\end{align}
\end{thm}

\begin{proof}
We use the notation of the proof of Theorem~\ref{branchup1}.
In light of \eqref{eqn:injloc},
 it is enough to show that
\begin{align}
\label{qqiso}
\Gamma(X', i'_+{\cal O}_{Y'}\otimes_{{\cal O}_{X'}}
 {\cal W}^p_{X'})
\simeq
\bigoplus_{\lambda'\in\Lambda}
 A_{\frak{q}''}(\lambda')^{\oplus m(\lambda'\!,\,p)}.
\end{align}

Let us prove that 
\begin{align}
\label{l'iso}
W^p 
\simeq \bigoplus_{\lambda'\in \Lambda}
 F(\lambda')^{\oplus m(\lambda'\!,\,p)}
\end{align}
as $L'$-modules.
Let $F$ be an irreducible finite-dimensional $L'$-module such that
 $\Hom_{L'} (F, W^p)\neq 0$.
Then the Frobenius reciprocity shows
$\Hom_{\ovl{Q}\cap L'} (F,\,\bb{C}_\lambda\otimes V^p)\neq 0$.
Since $L'$ is connected, $F$ is irreducible as an $\frak{l}'$-module.
Hence the $\frak{l}'$-module $F$
 is written as the exterior tensor product $F_c \boxtimes F_n$
 for an irreducible $\frak{l}'_c$-module $F_c$ and
 an irreducible $\frak{l}'_n$-module $F_n$.
Since $\lambda$ is linear and unitary,
 Remark~\ref{pzero} implies that
 $\Bar{\frak{q}}\cap \frak{p}$ acts by zero on $\bb{C}_\lambda$.
Hence $\frak{l}'_n$ also acts by zero on $\bb{C}_\lambda$.
Moreover, Lemma~\ref{unisub} implies that 
 $\frak{l}'_n$ acts by zero on $\frak{g}/(\Bar{\frak{q}}+\frak{g}')$.
Therefore, $\frak{l}'_n$ acts by zero on $W^p$. 
As a consequence, $F_n$ must be the trivial representation
 and $F\simeq F(\lambda')$ for some $\lambda'\in\Lambda$.
Then the Frobenius reciprocity gives
\[m(\lambda',p):=
\dim{\Hom}_{\ovl{Q}\cap L'} (F(\lambda'),\bb{C}_\lambda\otimes V^p)=
\dim{\Hom}_{L'} 
(F(\lambda'),\, W^p),
\]
and hence \eqref{l'iso} is proved.

We set 
\begin{align*}
&X''=G'/\ovl{Q''}, \quad 
Y''=K'/(\ovl{Q''}\cap K'), \\
&\xymatrix{
Y'' \ar[r]^*+{i''} \ar[d]_*+{} 
& X'' \ar[d]^*+{\pi'}\\
Y'   \ar[r]^*+{i'}   & X'
}
\end{align*}
where the maps are defined canonically.
By the same argument as in the proof of Lemma~\ref{square},
 we can prove that this diagram is a Cartesian square.
Take $\lambda'\in \Lambda$ 
and write ${\cal L}_{\lambda',X''}$ for
 the ${\cal O}_{X''}$-module associated with
 the $\ovl{Q''}$-module $\bb{C}_{\lambda'}$.
Theorem~\ref{locZuc} shows that
\begin{align}
\label{qloc}
A_{\frak{q}''}(\lambda')\simeq
\Gamma(X'', i''_+{\cal O}_{Y''}
\otimes_{{\cal O}_{X''}} {\cal L}_{\lambda',X''}).
\end{align}
As in the proof of Theorem~\ref{branchup1},
 we see that
\[
\pi'_*(i''_+{\cal O}_{Y''}
 \otimes_{{\cal O}_{X''}} {\cal L}_{\lambda',X''})
\simeq
i'_+{\cal O}_{Y'}\otimes_{{\cal O}_{X'}}
\pi'_*({\cal L}_{\lambda',X''}).
\]
Put $S':= \ovl{Q'}/\ovl{Q''}$ and
 write ${\cal L}_{\lambda',S'}$ for the ${\cal O}_{S'}$-module
 associated with $\bb{C}_{\lambda'}$.
The decompositions 
\[
\ovl{\frak{q}'}=\frak{l}'_c\oplus\frak{l}'_n
 \oplus\ovl{\frak{u}'},\quad
\ovl{\frak{q}''}=\frak{b}(\frak{l}'_c)\oplus\frak{l}'_n
 \oplus\ovl{\frak{u}'}
\]
 show that $S'$ is isomorphic to
 the complete flag variety of the reductive Lie algebra
 $\frak{l}'_c$.
Hence by the Borel--Weil theorem,
 $\Gamma(S', {\cal L}_{\lambda',S'})\simeq F(\lambda')$.
Then it follows from Lemma~\ref{push} that
\begin{align*}
\pi'_*({\cal L}_{\lambda',X''})\simeq 
 {\cal F}(\lambda')_{X'},
\end{align*} 
where
 ${\cal F}(\lambda')_{X'}$ is the ${\cal O}_{X'}$-module
 associated with the $\ovl{Q'}$-module $F(\lambda')$.
As a consequence,
 we have
\begin{align}
\label{gloiso}
\Gamma(X'', i''_+{\cal O}_{Y''}
 \otimes_{{\cal O}_{X''}} {\cal L}_{\lambda',X''})
&\simeq
\Gamma(X', \pi'_*(i''_+{\cal O}_{Y''}
 \otimes_{{\cal O}_{X''}} {\cal L}_{\lambda',X''}))\\ \nonumber
&\simeq
\Gamma(X', i'_+{{\cal O}_{Y'}}\otimes_{{\cal O}_{X'}}
 {\cal F}(\lambda')_{X'}).
\end{align}
The isomorphism \eqref{qqiso} follows from
 \eqref{l'iso}, \eqref{qloc}, and \eqref{gloiso}.
\end{proof}

\begin{rem}
{\rm
On the right side of \eqref{eqn:brinj2},
 $\lambda'$ may not be in the weakly fair range
 even if $m(\lambda',p)>0$.}
\end{rem}


\section{Associated Varieties}
\label{sec:ass}
As a corollary to Theorem~\ref{branchup2},
 we determine the associated variety of $(\frak{g}',K')$-modules
 that occur in $A_\frak{q}(\lambda)|_{(\frak{g}',K')}$.

For a finitely generated $\frak{g}$-module $V$,
 write $\Ass_\frak{g}(V)$ for the associated variety of $V$.
See \cite{kob98ii}, \cite{Vo91} for the definition.
We use the following fact on associated varieties.

\begin{fact}[\cite{kob98ii}]
\label{ass}
Let $\frak{g}$ be a complex reductive Lie algebra.
\begin{enumerate}
\item[{\rm (1)}]
$\Ass_{\frak{g}} (V)=\Ass_{\frak{g}} (V\otimes F)$ for
 any finitely generated $\frak{g}$-module $V$ and
 a nonzero finite-dimensional representation $F$ of $\frak{g}$.
\item[{\rm (2)}]
If $\lambda$ is in the weakly fair range and $A_\frak{q}(\lambda)$ is nonzero,
 then $\Ass_{\frak{g}}(A_\frak{q}(\lambda))=
 \Ad(K)(\Bar{\frak{u}}\cap \frak{p})$.
Here, we identify $\frak{g}$ with $\frak{g}^*$ by a non-degenerate 
 invariant bilinear form.
\end{enumerate}
\end{fact}

Fact~\ref{ass}~(2) can be generalized in the following way.
\begin{prop}
\label{aqass}
Let $\frak{q}$ be
 a $\theta$-stable parabolic subalgebra of $\frak{g}$ and
 $\bb{C}_\lambda$ a one-dimensional
 $(\frak{l},L\cap K)$-module.
Suppose that $V$ is an irreducible $(\frak{g},K)$-submodule
 of $A_\frak{q}(\lambda)$.
Then $\Ass_{\frak{g}} (V)= \Ad(K)(\Bar{\frak{u}}\cap \frak{p})$.
\end{prop}
\begin{proof}
If we take sufficiently large integer $N\in\bb{N}$,
 then $\lambda+2N\rho(\frak{u})$ is in the good range.
In view of Fact~\ref{ass}~(2), it is enough to show that
 $\Ass_\frak{g}(V)=\Ass_\frak{g}(A_\frak{q}(\lambda+2N\rho(\frak{u})))$.
Let $F$ be the irreducible finite-dimensional
 $(\frak{g},K)$-module with lowest weight $-2N\rho(\frak{u})$.
Then there is an injective $(\Bar{\frak{q}},L\cap K)$-homomorphism
 $\bb{C}_\lambda\to F\otimes \bb{C}_{\lambda+2N\rho(\frak{u})}$, which gives
 a long exact sequence:
\begin{align*}
&\qquad\qquad\cdots\to
(P_{\Bar{\frak{q}},L\cap K}^{\frak{g},K})_{s+1} 
((F\otimes \bb{C}_{\lambda+2N\rho(\frak{u})})/\bb{C}_\lambda)\to\\
\to
(P_{\Bar{\frak{q}},L\cap K}^{\frak{g},K})_s (\bb{C}_\lambda)
&\to 
(P_{\Bar{\frak{q}},L\cap K}^{\frak{g},K})_s 
 (F\otimes \bb{C}_{\lambda+2N\rho(\frak{u})})
\to\cdots.
\end{align*}
We claim that 
$(P_{\Bar{\frak{q}},L\cap K}^{\frak{g},K})_{s+1} 
((F\otimes \bb{C}_{\lambda+2N\rho(\frak{u})})/\bb{C}_\lambda)=0$.
Indeed, $(F\otimes \bb{C}_{\lambda+2N\rho(\frak{u})})/\bb{C}_\lambda$
 admits a finite filtration $\{F_p\}$ of
 $(\Bar{\frak{q}},L\cap K)$-modules such that
 $\Bar{\frak{u}}$ acts by zero on $F_p/F_{p-1}$.
Then \cite[Theorem 5.35]{KnVo} shows that 
 $(P_{\Bar{\frak{q}},L\cap K}^{\frak{g},K})_{s+1}(F_p/F_{p-1})=0$.
By using the exact sequences 
\[
(P_{\Bar{\frak{q}},L\cap K}^{\frak{g},K})_{s+1}(F_{p-1})
\to 
(P_{\Bar{\frak{q}},L\cap K}^{\frak{g},K})_{s+1}(F_{p})
\to
(P_{\Bar{\frak{q}},L\cap K}^{\frak{g},K})_{s+1}(F_p/F_{p-1})
\]
 iteratively, we can see that
 $(P_{\Bar{\frak{q}},L\cap K}^{\frak{g},K})_{s+1} 
((F\otimes \bb{C}_{\lambda+2N\rho(\frak{u})})/\bb{C}_\lambda)=0$.

As a result, we get an injective map
\begin{align*}
V\subset A_\frak{q}(\lambda)\to 
(P_{\Bar{\frak{q}},L\cap K}^{\frak{g},K})_s 
 (F\otimes \bb{C}_{\lambda+2N\rho(\frak{u})})
\simeq F\otimes A_\frak{q}(\lambda+2N\rho(\frak{u})),
\end{align*}
where the last isomorphism is the Mackey isomorphism
 \cite[Theorem 2.103]{KnVo}.
Then Fact~\ref{ass}~(1) shows that
\[
{\Ass}_{\frak{g}}(V)\subset 
{\Ass}_{\frak{g}}\bigl(F\otimes
 A_\frak{q}(\lambda+2N\rho(\frak{u}))\bigr)=
{\Ass}_{\frak{g}}\bigl(A_\frak{q}(\lambda+2N\rho(\frak{u}))\bigr).
\]
For the opposite inclusion,
 we see that
\[
{\Hom}_{\frak{g},K}\bigl(V\otimes F^*, 
 A_\frak{q}(\lambda+2N\rho(\frak{u}))\bigr)
\simeq
{\Hom}_{\frak{g},K}\bigl(V, 
 F\otimes A_\frak{q}(\lambda+2N\rho(\frak{u}))\bigr)\neq 0.
\]
Since $A_\frak{q}(\lambda+2N\rho(\frak{u}))$ is irreducible,
 there exists a surjective map
 $V\otimes F^* \to A_\frak{q}(\lambda+2N\rho(\frak{u}))$.
Therefore, 
Fact~\ref{ass}~(1) shows that
\[
{\Ass}_{\frak{g}}(V)= 
{\Ass}_{\frak{g}}(V\otimes F^*)\supset
{\Ass}_{\frak{g}}\bigl(A_\frak{q}(\lambda+2N\rho(\frak{u}))\bigr).
\]
Consequently,
\[
{\Ass}_{\frak{g}}(V)= 
{\Ass}_{\frak{g}}\bigl(A_\frak{q}(\lambda+2N\rho(\frak{u}))\bigr)
=\Ad (K)(\Bar{\frak{u}}\cap \frak{p}).
\]
\end{proof}

\begin{rem}
{\rm
In some literature, $A_\frak{q}(\lambda)$ is defined by
 using the derived functor of $I_{\frak{q}, L\cap K}^{\frak{g},K}$.
If we adopt this definition, 
 we have to replace `irreducible $(\frak{g},K)$-submodule' in
 Proposition~\ref{aqass} by `irreducible quotient $(\frak{g},K)$-module'.
Both definitions agree if $\lambda$ is unitary and in the weakly fair range.
}
\end{rem}

A connection between branching laws of $\frak{g}$-modules
 and their associated varieties was studied in \cite{kob98ii}.
\begin{fact}[{\cite[Theorem 3.1]{kob98ii}}]
Let $\frak{h}$ be a reductive Lie subalgebra of $\frak{g}$.
Write ${\pr}_{\frak{g}\to\frak{h}}:\frak{g}^*\to \frak{h}^*$ for 
 the restriction map.
Suppose that $W$ is an irreducible $\frak{g}$-module
 and $V$ is an irreducible $\frak{h}$-module such that
 ${\Hom}_{\frak{h}}(V,W)\neq 0$.
Then 
\[
{\pr}_{\frak{g}\to\frak{h}}({\Ass}_{\frak{g}}(W))\subset
{\Ass}_{\frak{h}}(V).
\]
\end{fact}

In our setting, we can deduce from Theorem~\ref{branchup2}
 that the equality holds.

\begin{thm}
\label{branchass}
Let the notation and the assumption be as in Theorem~\ref{branchup1}.
Suppose that $V$ is an irreducible $(\frak{g}',K')$-module
 such that ${\Hom}_{\frak{g}'}(V, A_\frak{q}(\lambda))\neq 0$.
Then
\[
{\pr}_{\frak{g}\to\frak{g}'}({\Ass}_{\frak{g}}(A_\frak{q}(\lambda)))=
{\Ass}_{\frak{g}'}(V).
\]
\end{thm}
\begin{proof}
In light of Theorem~\ref{branchup2},
 we see that $V$ is isomorphic to
 an irreducible $(\frak{g}',K')$-submodule of
 $A_{\frak{q}''}(\lambda')$ for some character $\lambda'$.
Then Proposition~\ref{aqass} and Fact~\ref{ass}~(2) show that
\[{\Ass}_{\frak{g}'}(V)=\Ad (K')(\ovl{\frak{u}''}\cap \frak{p}'),\quad
{\Ass}_{\frak{g}}(A_\frak{q}(\lambda))=\Ad (K)(\ovl{\frak{u}}\cap \frak{p}).
\]
Therefore, it is enough to prove that
\[
{\pr}_{\frak{g}\to\frak{g}'}(\Ad (K)(\frak{u}\cap \frak{p}))
=\Ad (K')(\frak{u}''\cap\frak{p}').
\]
Since $\frak{q}$ is $\sigma$-open,
 $K'/(Q\cap K')$ is open dense in the partial flag variety
 $K/(Q\cap K)$. 
As a result, $\Ad (K')(\frak{u}\cap \frak{p})$ is dense in
 $\Ad (K)(\frak{u}\cap \frak{p})$ and hence
 ${\pr}_{\frak{g}\to\frak{g}'}(\Ad (K')(\frak{u}\cap \frak{p}))$
 is dense in ${\pr}_{\frak{g}\to\frak{g}'}(\Ad (K)(\frak{u}\cap \frak{p}))$.
From the proof of Proposition~\ref{parab}, we have
\[
{\pr}_{\frak{g}\to\frak{g}'}(\frak{u}\cap \frak{p})
=\frak{u}'\cap\frak{p}'=\frak{u}''\cap \frak{p}'.
\]
Consequently, $\Ad(K')(\frak{u}''\cap\frak{p}')$ is a dense
 subset of ${\pr}_{\frak{g}\to \frak{g}'}(\Ad (K)(\frak{u}\cap \frak{p}))$.
Since $\Ad(K')(\frak{u}''\cap\frak{p}')$ is closed,
 we conclude that
\[
{\pr}_{\frak{g}\to\frak{g}'}(\Ad (K)(\frak{u}\cap \frak{p}))
=\Ad (K')(\frak{u}''\cap\frak{p}'),
\]
which completes the proof.
\end{proof}

\end{document}